\documentclass[11pt]{amsart}

\usepackage[all]{xy}
\usepackage{amsfonts,amssymb,amsmath,amscd}
\usepackage{amsthm}

\usepackage{mathrsfs}

\swapnumbers
\newtheorem{theorem}{Theorem}[section]
\newtheorem{lemma}[theorem]{Lemma}
\newtheorem{thm}[theorem]{}
\newtheorem{proposition}[theorem]{Proposition}

\newtheorem{corollary}[theorem]{Corollary}
\newtheorem{remark}[theorem]{Remark}
\numberwithin{equation}{section}

\newcommand{\bF}{{\textsf{F}}}
\newcommand{\bG}{{\textsf{G}}}
\newcommand{\bH}{{\textsf{H}}}
\newcommand{\bT}{{\textsf{T}}}
\newcommand{\bS}{{\mathsf{S}}}

\newcommand{\bP}{{\mathsf{P}}}

\newcommand{\I}{{\mathbb{I}}}

\newcommand{\Aa}{{\A^{\overline{\bH}}_{\underline{\bH}}\,(\omega)}}

\newcommand{\A}{\mathbb{A}}
\newcommand{\B}{\mathbb{B}}
\newcommand{\C}{\mathbb{C}}
\newcommand{\M}{\mathbb{M}}
\newcommand{\X}{\mathbb{X}}
\newcommand{\Y}{\mathbb{Y}}
\newcommand{\V}{\mathbb{V}}

\newcommand{\xr}{\xrightarrow}
\newcommand{\ul}{\underline}
\newcommand{\ol}{\overline}
\newcommand{\wt}{\widetilde}

\newcommand{\Ab}{{(\A_{\underline{\bH}})^{\bG}}}
\newcommand{\Ac}{{(\A^{\ol{\bH}})_\bT}}

\newcommand{\xra}{\xrightarrow}

\newcommand{\ve}{\varepsilon}

\newcounter{zlist}
\newenvironment{zlist}{\begin{list}{\rm(\arabic{zlist})}{
\usecounter{zlist}\leftmargin2.5em\labelwidth2em\labelsep0.5em
\topsep0.6ex\itemsep0.3ex plus0.2ex minus0.3ex
\parsep0.3ex plus0.2ex minus0.1ex}}{\end{list}}

\newcounter{blist}
\newenvironment{blist}{\begin{list}{{\rm(\alph{blist})}}{
\usecounter{blist}\leftmargin2.5em\labelwidth2em\labelsep0.5em
\topsep0.6ex \itemsep0.3ex plus0.2ex minus0.3ex
\parsep0.3ex plus0.2ex minus0.1ex}}{\end{list}}

\newcounter{rlist}

\newcommand{\ot}{\otimes}

\newcommand{\oxi}{{\overline{\xi}}}
\newcommand{\ochi}{{\overline{\chi}}}

\newcommand{\oom}{\overline{\omega}}

\def\tw{{\rm \textsf{tw}}}

\begin{document}

\title[Fundamental Theorem]{The Fundamental Theorem for \\ weak braided bimonads}

\author[Mesablishvili]{Bachuki Mesablishvili}
\address{A. Razmadze Mathematical Institute of I. Javakhishvili Tbilisi
State University, 6, Tamarashvili Str.,  Tbilisi 0177, Republic of Georgia}
\email{bachi@rmi.ge}
\author[Wisbauer]{Robert Wisbauer}
\address{Department of Mathematics of HHU, 40225 D\"usseldorf, Germany}
\email{wisbauer@math.uni-duesseldorf.de}

\begin{abstract}
The theories of (Hopf) bialgebras and  weak (Hopf) bialgebras have been introduced for
vector space categories over fields and make heavily use of the tensor product.
As first generalisations, these notions were formulated for monoidal categories,
with braidings if needed. The present authors developed a theory of bimonads and
Hopf monads $H$ on arbitrary categories $\A$, employing distributive laws,
allowing for a general form of the Fundamental Theorem for Hopf algebras.
For $\tau$-bimonads $H$,  properties of braided (Hopf) bialgebras were captured by
requiring a Yang-Baxter operator $\tau:HH\to HH$. The purpose of this paper
is to extend the features of weak (Hopf) bialgebras to this general setting
including an appropriate form of the Fundamental Theorem.
This subsumes the theory of braided Hopf algebras (based on weak Yang-Baxter operators) as considered by Alonso \'Alvarez and others.
\end{abstract}

\maketitle

\tableofcontents

\section{Introduction}
There are various generalisations of the notions of (weak) bialgebras
and Hopf algebras in the literature,
mainly for (braided) monoidal categories, and we refer to B\"ohm \cite{Bo-HA},
the introductions to Alonso \'Alvarez e.a. \cite{AVR},   B\"ohm e.a. \cite{BLS}, and \cite[Remarks 36.18]{BW-Cor}
for more information about these.

Bimonads and Hopf monads on {\em arbitrary} categories were introduced in \cite{MW-Bim}
and the purpose of the present paper is to develop  a {\em weak version} of these notion, that is,
the initial conditions on the behaviour  of the involved distributive laws
towards unit and counit are replaced by weaker conditions.

Recall that for a bialgebra  $(H,m,e,\delta,\ve)$ over a commutative ring $k$,
there is a commutative diagram ($\ot_k=\ot$)
$$\xymatrix{
 \M \ar[rr]^{-\ot H} \ar[dr]_{\phi_H} &&  {\M_H^H} \ar[dl]^{U^{\wt H}} &  \\
 & {\M_H}&, } \hspace{-4mm}
  {\scriptstyle \xymatrix{
 M\ar@{|->}[r] \ar@{|->}[dr] &
          (M\ot H, M\ot m ,M\ot \delta  ) \ar@{|->}[d]  , \\
&(M\ot H ,M\ot m ) \, ,} }
$$
where $\M$ is the category of $k$-modules,
 $ \M_H$ the category of right $H$-modules, and
$\M^H_H$ denotes the category of mixed bimodules; the latter can also be
considered as $(\M_H)^{\wt H}$, that is, the category of $\wt H$-comodules over
$\M_H$ where $\wt H$ is the lifting of the comonad $-\ot H$ to $\M_H $.
$H$ is a Hopf algebra provided the functor $ -\ot H$ is an equivalence of categories
(Fundamental Theorem of Hopf modules).

Concentrating on the essential parts of this setting, we consider, for any category $\A$,
 the  diagram
\begin{equation}\label{dia-1a}
\xymatrix{
 \A \ar[rr]^{K} \ar[dr]_{\phi_\bT} &&  ({\A_\bT})^\bG \ar[dl]^{U^{\bG}} &
   \\
 & {\A_\bT},    }
\end{equation}
where $\bT$ is some monad on $\A$, $\bG$ is some comonad on the category $\A_\bT$  of $\bT$-modules,
$\phi_\bT$ and $U^\bG$ denote the respective free and forgetful functors,
and $K$ is any functor making the diagram commutative.

Having such a diagram, one may ask when the functor $K$ allows for a right adjoint
$\ol K$.
This creates a monad $\bP$ on $\A$, a monad morphism $\iota:P\to T$,
the free functor $\phi_\bP:\A \to \A_\bP$, and a functor
$K_\bP:(\A_\bT)^{\bG} \to \A_\bP$, the Eilenberg-Moore comparison functor for the monad $P$.

 If $\A$ is Cauchy complete and $\bP$ is a separable Frobenius monad, then the
 change-of-base-functor $\iota_!: \A_\bP\to \A_\bT$ exists (see Proposition \ref{sep.frobenius}).
  As a consequence, $K_\bP$ has a left adjoint $L_\bP$ (Proposition \ref{left.adjoint}) leading to the commutative diagram
\begin{equation}\label{diagram.2a}
\xymatrix@C=.6in @R=.2in {\A \ar[r]_{\phi_\bP}\ar@/^1.8pc/@{->}[rr]^{K}\ar[rdd]_{\phi_\bT}  &
\A_{\bP}  \ar@{..>}[r]_{L_\bP} \ar@{..>}[dd]_{\iota_!}&
(\A_\bT)^{\bG}  \ar[ldd]^{U^{\bG}}\\ &&\\  &  \A_\bT  &. }
\end{equation}
 So asking for a fundamental theorem leads to the question for properties
of $\phi_\bP$ and $K_\bP$.
Similar constructions apply when the functor $K$ has a left adjoint.
\smallskip

After assembling preliminaries in Section \ref{Prel} and properties of
separable Frobenius monads in Section \ref{Frob}, the theory just sketched is outlined in Section \ref{Com}.

Section \ref{Entw} deals with the application of this to
endofunctors $H$ on a category $\A$ endowed with a monad $\ul H$ 
as well as a comonad structure $\ol H$,
and a weak  monad-comonad entwining $\omega:  HH\to H H$.
Exploiting ideas from  \cite{W-reg} and B\"ohm \cite{Bo},
these data allow for the definition of a comonad $\bG$ on $\A_H$
as well as a monad $\bT$ on $\A^H$ (Propositions \ref{Bohm}, \ref{Bohm1}).
Hereby, the category $\A_{\ul \bH}^{\ol \bH}(\omega)$ of mixed $H$-bimodules is isomorphic to the
categories $(\A_H)^\bG$ and $(\A^H)_\bT$ (see Theorems \ref{comparison}, \ref{comparison1}).
If $\omega$ is a {\em compatible} entwining (i.e. $\delta \cdot m= Hm \cdot \omega H\cdot H\delta$,
Section \ref{w-entw}), there is a functor
$$K_{\omega}:\A \to \A^{\ol \bH}_{\ul \bH}(\omega),\quad a \; \mapsto \; (H(a),m_a, \delta_a),$$
leading to commutativity of the diagram corresponding to (\ref{dia-1a}). Conditions for the existence
of a right and a left adjoint functor for $K_\omega$ are investigated
(Pro\-positions \ref{pair}, \ref{pair1}).   These problems were considered in \cite{MW-Bim}
for proper compatible entwinings $\omega$.

In Section \ref{t-bim}, we
define {\em weak $\tau$-bimonads}, also called {\em weak braided bimonads} (Definition \ref{def-bialg}),
based on a {\em weak Yang-Baxter operator}
$\tau:HH\to HH$ (Section \ref{weak-YB}). This type of operator
was introduced by Alonso \'Alvarez e.a.  in 
\cite{AV-JA}  for monoidal categories and is here adapted to the more general setting.
The conditions on weak braided bimonads
 induce a weak monad-comonad entwining $\omega:HH\to HH$ as well as a weak
comonad-monad entwining $\oom:HH\to HH$
and allow to refine the results in Section \ref{Entw}:
the natural transformation $\oxi:=H\ve\cdot \oom\cdot He: H\to H$  is idempotent
 and its splitting yields a  separable Frobenius monad $H^\oxi$
(see Proposition \ref{sep.frobenius2}).
 Then, if $K_\omega$ has  a right adjoint functor $\ol K$,
the induced monad $P$ is just $H^\oxi$
 and the  diagram corresponding to  (\ref{diagram.2a}) can be completed.

The {\em weak bialgebras} over a commutative ring $k$ as considered
by B\"ohm e.a. in \cite{BCJ}  are weak braided bimonads in our sense
($\tau$ the ordinary twist, $\oxi=\ve_s$) and for this case some of our results
are shown there, including the Frobenius and separability property of
$H^\oxi (=H_s)$  (\cite[Proposition 4.4]{BCJ}).

Eventually, in Section \ref{t-Hopf}, {\em weak braided Hopf monads} are defined as
weak braided bimonads $\bH$ with a {\em weak antipode} (Definition \ref{def-S}).
 In monoidal categories, these correspond to the weak braided Hopf algebras
considered  in \cite{AVR}, \cite{AVR-ent}.

We show  that for a braided bimonad
in Cauchy complete categories the existence of an antipode
is equivalent to the functor $K_\omega$ having a left adjoint and a monadic right
adjoint and this leads to an equivalence between the
categories of $H^\oxi$-modules and $\A^{ol \bH}_{\ul\bH}(\omega)$ (Fundamental Theorem \ref{main}).

Examples for our weak braided bimonads and
weak braided Hopf monads are the {\em weak braided  Hopf
algebras}  in strict monoidal Cauchy complete categories considered by
Alonso \'Alvarez et al. in \cite{AV-JA,AVR, AVR-ent},
the {\em weak bimonoids} and {\em weak Hopf monoids} in braided monoidal categories
as defined by Pastro et. al. in \cite{PaSt} and also showing up in \cite[Sections 3,4]{BLS}.
These all subsume the braided Hopf algebras considered, e.g., by Takeuchi \cite{Tak}
and Schauenburg \cite{Sch} and, of course,  the
weak Hopf algebras in vector space categories introduced by B\"ohm et al. in \cite{BNS}.
Moreover, we generalise  the {\em bimonads} and {\em $\tau$-Hopf monads} defined on arbitrary categories in \cite{MW-Bim}
and these include, for example, bimonoids in duoidal categories
(e.g. \cite{MW-Gal-Gen}).

 {\em Opmonoidal monads $\bT=(T,m,e)$} on strict  monoidal categories $(\V,\ot, \I)$
were also called {\em bimonads} by Brugui\`eres et al. in \cite {BrVi} and these were
generalised to {\em weak bimonads} in monoidal categories by B\"ohm et al. in
\cite{BLS}.
As pointed out in \cite[Section 5]{MW}, the bimonads from \cite {BrVi} yield a special case of an entwining of
the monad $T$ with the comonad $-\ot T(\I)$, where $T(\I)$
has a coalgebra structure derived from the opmonoidality of $\bT$.
To transfer the structures from  \cite{BLS} to arbitrary categories one has to
consider {\em weak entwinings} between monads and comonads.
It is planned to elaborate details for this in a subsequent article.


\section{Preliminaries}\label{Prel}

\begin{thm} \label{mon-comon}{\bf Monads and comonads.}  \em Recall that a {\em monad} $\bT$ on a category $\A$ (or
shortly an $\A$-monad $\bT$) is a triple $(T,m,e)$ where $T:\A\to \A$ is a functor with natural
transformations $m:TT\to T$, $e:1\to T$ satisfying the usual associativity and unitality conditions. A $\bT${\em -module}
is an object $a\in \A$ with a morphism $h:T(a)\to a$ subject to associativity and unitality
conditions. The (Eilenberg-Moore) category of $\bT$-modules is denoted by $\A_\bT$ and there is an adjunction
$$e_\bT, \ve_\bT:\phi_\bT \dashv U_\bT: \A_\bT \to \A,$$
with $\phi_\bT:\A\to \A_\bT$ and
$U_\bT:\A_\bT\to \A$  given by the respective assignments
$$ a \mapsto (T(a),m_a) \,\,\, \text{and}\,\,\,(a, h) \mapsto a,$$
while $e_\bT=e$ and $(\ve_\bT)_{(a,\,h)}=h$ for each $(a,h)\in \A_\bT$.

If $\bT=(RF, R \ve F,\eta)$ is the monad generated on $\A$ by an adjoint pair
$ \eta , \ve \colon F \dashv R \colon \B \to \A$, then there is the \emph{comparison functor}
$K_{\bT} : \B \to \A_{\bT}$ which
assigns to each object $ b \in \B$ the $\bT$-algebra $ (R(b), R(\ve_b)),$ and to each morphism $ f : b  \to b'$
the morphism $R(f): R(b) \to R(b'),$ satisfying $U_{\bT} K_{\bT} = R$ and $K_{\bT} F = \phi_{\bT}$. This
situation is illustrated by the diagram
$$
\xymatrix@!=2.4pc{ \B \ar@<.6ex> [dr]^R \ar[rr]^{K_\bT} && \A_\bT \ar@<-.6ex> [dl]_{U_\bT}\\
& \A \ar@<.6ex> [ul]^F \ar@<-.6ex> [ru]_{\phi_\bT}& .}
$$

The functor $R$ is called \emph{monadic} (resp. \emph{premonadic}) if the comparison functor $K_{\bT}$ is an equivalence
of categories (resp. full and faithful).

\begin{theorem} \label{Beck}\emph{(Beck, \cite{Be})} Let $ \eta, \epsilon \colon F \dashv R \colon \B \to \A$
be an adjunction, and $\bT=(RF, R\ve F, \eta)$ the corresponding monad on $\A$. 

\begin{zlist}
\item  The comparison functor $K_{\bT} : \B \to \A_{\bT}$ has a left adjoint $L_{\bT} : \A_{\bT} \to \B$ if and
only if for each $(a, h) \in \A_{\bT}$, the pair of morphisms $(F(h), \ve_{F(a)})$ has a coequaliser in $\B$.
\item  $R$ is monadic if and only if it is conservative and for  $(a, h) \in \A_\bT$, the pair
of morphisms $(F(h), \ve_{F(a)})$ has a co\-equaliser and this coequaliser is preserved by $R$.
\end{zlist}
\end{theorem}

Dually, a {\em comonad} $\bG$ on $\A$ (or shortly an $\A$-comonad $\bG$) is a triple $(G,\delta,\ve)$ where $G:\A\to \A$ is a functor with natural
transformations $\delta:G\to GG$, $\ve:G\to 1$, and $\bG$-{\em comodules} are objects $a\in \A$ with morphisms
$\theta:a\to G(a)$. Both notions are subject to coassociativity and counitality conditions. The (Eilenberg-Moore)
category of $\bG$-comodules is denoted by $\A^\bG$ and there is a cofree functor
$$\phi^\bG:\A\to \A^\bG,\; a\mapsto (G(a),\delta_a),$$ which is right adjoint to the forgetful functor
$$U^\bG:\A^\bG\to \A, \; (a,\theta) \mapsto a.$$

If $\eta, \ve: F \dashv R \colon \A \to \B$ is an adjoint pair and
$\bG=(FR, F \eta R,\ve)$ is the comonad on $\A$
associated to $(R, F)$, then one has the comparison functor
$$K^\bG: \B \to \A^\bG, \, \, b \to (F(b), F(\eta_b))$$ for which $U^\bG \cdot K^\bG= F$ and $K^\bG \cdot R = \phi^\bG$.
One says  that the functor $F$ is \emph{precomonadic} if $K^\bG$ is full and faithful, and it is \emph{comonadic} if
$K^\bG$ is an equivalence of categories.
\end{thm}

\begin{thm} \label{Cauchy}{\bf Cauchy completeness.}  \em A morphism $e:A \to A$ in $\A$
is \emph{idempotent} if $e^2 = e$ and $\A$ is said to be \emph{Cauchy complete} if for every idempotent
$e:a \to a$, there exists an object $\overline{a}$ and morphisms $p : a \to \overline{a}$ and
$i:\overline{a} \to a$ such that $ip = e$ and $pi = 1_{\overline{a}}$. In this case, $(\overline{a}, i)$
is the equaliser of $e$ and $1_a$ and $(\ol a, p)$ is the coequaliser of $e$ and $1_a$. Hence any category
admitting either equalisers or coequalizers is Cauchy complete.

\begin{proposition} \label{idem}Let $\bG$ be a comonad on $\A$. If $\A$  is Cauchy complete,
then so is $\A^\bG$. Moreover, the forgetful functor $U^\bG: \A\!^\bG \to \A$
preserves and creates splitting of idempotents. Explicitly, if $e: (a, \theta) \to (a, \theta)$
is an idempotent morphism in $\A\!^\bG$ and if $a \xr{p} \overline{a} \xr{i} a$ is a splitting of $e$ in $\A$,
then $(\overline{a}, G(p) \cdot \theta \cdot i)$ is a $\bG$-comodule in such a way that $p$ and $i$ become
morphisms of $\bG$-comodules. Similarly, if $\bT$ is a monad on $\A$, then the forgetful functor
$U_\bT: \A_\bT \to \A$ preserves and creates splitting of idempotents.
\end{proposition}
\begin{proof} The result follows from the fact that the forgetful functor
 $U^\bG: \A^\bG \to \A$
preserves and creates coequalisers, while the functor $U_\bT: \A_\bT \to \A$ preserves
 and creates equalisers.
\end{proof}
\end{thm}

\begin{thm} \label{contractible}{\bf Split (co)equalisers.}  \em Recall
(e.g. from \cite{Mc})  that a diagram
\begin{equation}\label{contractible.d}
\xymatrix{a  \ar@{->}@<0.5ex>[r]^-{\partial_0}
\ar@{->}@<-0.5ex>[r]_{\partial_1}& a' \ar[r]_p& x }
\end{equation} with $p \partial_0=p\partial_1$ is said to be \emph{split} by a pair of morphisms
$i:x \to a'$ and $s: a' \to a$ if $pi=1_x, \,\partial_0s=1_{a'}$ and $\partial_1 s=ip$.

A pair of morphisms  $(\partial_0,\partial_1:a \rightrightarrows a')$ in  $\A$ is called \emph{split} if
there exists a morphism  $ s:a' \to a$ with $ \partial_0 s=1$ and $\partial_1 s \partial_0=\partial_1 s \partial_1$.
In this case, $\partial_1 s: a' \to a'$ is an idempotent, and if we assume  $\A$ to be Cauchy complete
and  if $a' \xr{p}x \xr{i} a'$ is a splitting of the idempotent $\partial_1 s$, then the diagram
$$
\xymatrix{a  \ar@{->}@<0.5ex>[rr]^-{\partial_0}
\ar@{->}@<-0.5ex>[rr]_{\partial_1}&& a' \ar@/_1.6pc/ [ll]_{s} \ar[r]_p& x
 \ar@/_1.6pc/ [l]_{i}}
$$ is a split coequaliser. Conversely, if the
above diagram is a split coequaliser, then $s$ makes the pair $(\partial_0,\partial_1:a
\rightrightarrows a')$ split. Thus, when $\A$ is Cauchy complete, a pair $(\partial_0,\partial_1:a
\rightrightarrows a')$ is part of a split coequaliser diagram if and only if it is split. Note
that split (co)equalisers are \emph{absolute}, i.e., they are preserved by any functor.

If $F: \A \to \B$ is a functor, $(\partial_0,\partial_1:a \rightrightarrows a')$ is called
\emph{$F$-split} if the pair of morphisms $(F(\partial_0), F(\partial_1) :F(a)
\rightrightarrows F(a'))$ in $\B$ is split.

The dual notions are those of \emph{cosplit pairs} and $F$-\emph{cosplit pairs}.
\end{thm}

Given a monad $\bT$ (resp. comonad $\bG$) on $\A$ and a category $\X$, one may consider the functor
category $[\X, \A]$ and the induced monad $[\X, \bT]$ (resp. comonad $[\X, \bG]$) thereon, whose
functor part sends a functor $F: \X \to \A$ to the composite $TF: \X \to \A$ (resp. $GF: \X \to \A$). Symmetrically, one
has the induced monad $[\bT, \X]$ (resp. comonad $[\bG, \X]$) on $[\A, \X]$, whose functor-part sends
$F: \A \to \X$ to $FT: \A \to \X$ (resp. $FG: \A \to \X$).

\begin{thm} \label{bimodule.funct.}{\bf (Bi)module functors.} \em Given a monad $\bT=(T, m,e)$ on $\A$ and
a category $\X$, a \emph{left} $\bT$-\emph{module with domain} $\X$ is an object of the Eilenberg-Moore
category $[\X,\A]_{[\X, \bT]}$ of the monad $[\X, \bT]$. Thus, a left $\bT$-module with domain $\X$ is
a functor $M : \X \to \A$ together with a natural transformation $\varrho : TM \to M$, called the
\emph{action} (or the $\bT$-\emph{action}) on $M$, such that $\varrho \cdot eM=1$ and $\varrho \cdot T\varrho=\varrho \cdot mM$.
A morphism of left $\bT$-modules with domain $\X$ is a natural transformation in $[\X,\A]$ that commutes
with the $\bT$-actions.

Similarly, for a category $\Y$, the category of \emph{right} $\bT$-\emph{modules with codomain}
$\Y$  is defined as the Eilenberg-Moore category $[\A, \Y]_{[\bT, \Y]}$ of the monad $[\bT, \Y]$.

Let $\bS$ be another monad on $\A$. A $(\bT, \bS)$-
\emph{bimodule} is a functor $N: \A \to \A$ equipped with two natural transformations $\varrho_l: TN \to N$
and $\varrho_r: N S \to N$ such that $(N, \varrho_l) \in [\X,\A]_{[\X, \bT]}$, $(N, \varrho_r)\in [\A, \Y]_{[\bT, \Y]}$
and $\varrho_r \cdot \varrho_l S=\varrho_l \cdot T\varrho_r.$ A morphism of $(\bT,\bS)$-bimodules is a morphism
of left $\bT$-modules which is simultaneously a morphism of right $\bS$-modules. We write $[\X,\A]_{[\bS, \bT]}$
for the corresponding category.
\end{thm}

\begin{thm} \label{canonical module}{\bf Canonical modules.} \em Let $\bT=(T, m,e)$ be an arbitrary monad on $\A$.
Using the associativity and unitality of the multiplication $m$,  we find that for any functor $M:\X \to \A$, the pair
$(TM, mM)$ is a left $\bT$-module. Moreover, if $\nu : M \to M'$ is a natural transformation, then $T\nu: TM \to TM'$
is a morphism of left $\bT$-modules.

Symmetrically, for any functor $N:\A \to \Y$, the pair
$(NT, Nm)$ is a right $\bT$-module, and if $\nu : N \to N'$ is a natural transformation, then $\nu T: NT \to NT'$
is a morphism of right $\bT$-modules. In particular, $(T,m)$ can be regarded as a right as well
as a left $\bT$-module; again by the associativity of $m$, $(T,m,m)$ is a $(\bT,\bT)$-bimodule. Moreover, if
$\bS$ is another monad on $\A$ and $\iota: \bS \to \bT$ is a monad morphism, then
\begin{itemize}
  \item [(i)] $(T, ST \xr{\iota T}TT \xr {m} T)$ is a
left $\bS$-module;
  \item  [(ii)]$(T, TS \xr{T\iota}TT \xr {m} T)$ is a right $\bS$-module;
  \item [(iii)] $(T, TT \xr{m} T, TS \xr{T\iota}TT \xr {m} T)$ is a $(\bT,\bS)$-bimodule;
  \item [(iv)] $(T, ST \xr{\iota T}TT \xr {m} T, TT \xr{m} T)$ is an $(\bS,\bT)$-bimodule;
  \item [(v)] $(T, ST \xr{\iota T}TT \xr {m} T, TS \xr{T\iota}TT \xr {m} T)$ is an $(\bS,\bS)$-bimodule;
  \item [(vi)] $(\phi_\bT,  \ve_\bT \phi_\bT \cdot \phi_\bT \iota)$ is right $\bS$-module;
  \item [(vii)] $(U_\bT,   U_\bT\ve_\bT \cdot  \iota U_\bT )$ is left $\bS$-module.
\end{itemize}
\end{thm}

\begin{thm}\label{tensor}{\bf Tensor product of module functors.} \em
Let $\bT$ be a monad  on $\A$ and  $\X$ and $\Y$
arbitrary categories. If $(N,\rho)$ is a left $\bT$-module and $(M, \varrho)$ a right $\bT$-module, then
their \emph{tensor product} (\emph{over} $\bT$) is defined as the object part of the coequaliser
\begin{equation}\label{tensor1}
{\xymatrix @C=.5in
{M T N  \ar@{->}@<0.5ex>[r]^{\varrho N} \ar@ {->}@<-0.5ex> [r]_{M \rho}& M N
\ar[r]^-{\text{can}_\bT^{M,N}}&  M \! \otimes_\bT \!N}}
\end{equation} in $[\X, \Y]$, provided that such a coequaliser exists. We often abbreviate $\text{can}_\bT^{M,N}$
to $\text{can}^{M,N}$, or even to $\text{can}$.
\end{thm}

\begin{proposition} \label{action} Let $\bS, \bT$ be monads on a category $\A$. If the modules
$(M, \varrho)\in [\A , \X]_{[\bT,\X]}$ and  $(N, \rho_l, \rho_r) \in [\A,\A]_{[\bS, \bT]}$ are such that the
tensor product  $M \! \otimes_\bT \!N$ exists, then there
is a natural transformation
$$\zeta:(M  \otimes_\bT N)S \to M \! \otimes_\bT \!N $$
making  $M \! \otimes_\bT \!N$ into a right
$\bS$-module in such a way that the diagram
$$
\xymatrix@C=.5in
{(M T N, M T \rho_r)  \ar@{->}@<0.5ex>[r]^{\;\varrho N} \ar@ {->}@<-0.5ex>
[r]_{\;M \rho_l}& (M N, M \rho_r)
\ar[r]^-{\text{can}}& (M \! \otimes_\bT \!N, \zeta)} $$
is a coequaliser in $[\A, \X]_{[\bS, \X]}$.
The action $\zeta$ is uniquely determined by the property
$$\zeta \cdot \text{can}_\bT^{M,N} S=\text{can}_\bT^{M,N} \cdot M \rho_r.$$
\end{proposition}
\begin{proof} Note first that both $(M T N, M T \rho_r)$ and $(M N, M \rho_r)$ are objects
of the category $[\A, \X]_{[\bS, \X]}$, and that both $\varrho N$ and $M \rho_l$ can be
seen as morphisms in $[\A, \X]_{[\bS, \X]}$ (the former by naturality of composition, the latter because
$(N, \rho_l, \rho_r) \in [\A,\A]_{[\bS, \bT]}$). Now, since the functor $$[\bS, \X]: [\A, \X] \to [\A, \X]$$
preserves all colimits (hence coequalisers), the result follows from \cite[Proposition 2.3]{L}.
\end{proof}

\begin{thm}\label{comodules}{\bf Comodule functors.} \em Similarly as for monads, one defines  comodule functors  over a comonad $\bG$ on a category $\A$. A
{\em left $\bG$-comodule functor} is a pair $(F, \theta)$,
where $F: \X \to \A$ is a functor  and $\theta : F \to GF$ is a natural transformation inducing commutativity of the diagrams
$$\xymatrix{F \ar@{=}[dr] \ar[r]^-{\theta}& GF \ar[d]^-{\ve F}\\ & F,} \qquad \xymatrix{
F \ar[r]^-{\theta} \ar[d]_-{\theta}& GF \ar[d]^-{\delta F}\\
GF \ar[r]_-{G \theta}& GGF.}
$$
 Symmetrically,  {\em right $\bG$-comodule functors} $\A\to \Y$ are defined.
\end{thm}

\begin{thm}\label{cotensor}{\bf Cotensor product of comodule functors.} \em
Let $\bG$ be a comonad  on $\A$ and  $\X$, $\Y$ arbitrary categories.
If $(C,\theta)$ is a right $\bG$-comodule functor $\A\to \X$
and $(D, \vartheta)$ is a left $\bG$-comodule functor $\Y\to \A$, then their \emph{cotensor product} (\emph{over} $\bG$) is defined as
the object part of the equaliser
\begin{equation}\label{cotensor1}
\xymatrix{C \!\ot^{\bG}\! D\ar[rr]^-{\text{can}_\bG^{C,D}} && CD  \ar@{->}@<0.5ex>[r]^{\theta D} \ar@ {->}@<-0.5ex> [r]_{C \vartheta}& CGD
}\end{equation}
in $[\Y, \X]$, provided that such an equaliser exists. We often abbreviate $\text{can}_\bG^{C,D}$
to $\text{can}^{C,D}$, or even to $\text{can}$.
\end{thm}


Now recall the dual of \cite[Lemma 21.1.5]{S}) and Dubuc's
Adjoint Triangle Theorem  \cite{D}.

\begin{thm}\label{adj-monad}{\bf Monads and adjunctions.} \em
For categories $\A$, $\B$ and $\C$, consider the adjunctions
 $\eta, \ve : F \dashv R: \A \to \C$ and $\eta', \ve' : F' \dashv R': \B \to \C$,
 with corresponding monads $\bT$ and $\bT'$, and let $$ \xymatrix{\A \ar[rd]_{R} \ar[r]^{K}& \B \ar[d]^{R'}\\ &  \C}$$
be a commutative diagram with suitable functor $K$. Write $\gamma_K$ for the composite
$$F' \xr{F'\eta}F'RF=F'R'KF \xr{\ve' KF} KF.$$
Then $t_K=R' \gamma_K$ is a monad morphism $\bT' \to \bT$ (see \cite{D}).
\end{thm}
\begin{theorem}\label{Dubuc1} In the situation considered in \ref{adj-monad},
suppose that $R$ is a premonadic functor.
Then $K$ has a left adjoint $\ul{K}$ if and only if the following
coequaliser (used as the definition of $\ul{K}$) exists in $[\B, \A]$,
$$\xymatrix@!=2.6pc{FR'F'R' \ar[rrr]^{FR'\ve'} \ar[drr]|{Ft_KR'=FR'\gamma_{_K} R'}&& & FR' \ar[r]^-q & \ul{K} \\
           & & FR'KFR'=FRFR'  \ar[ru]|{\ve FR'} && . } $$
\end{theorem}
When this is the case, the unit $\ul{\eta}$ of the adjunction $\ul{K} \dashv K$ is the unique natural
transformation $1 \to K \ul{K}$ yielding commutativity of the diagram
$$ \xymatrix @R=.4in{F'R'\ar[d]_-{\gamma_{_K}\! R'} \ar[r]^-{\ve'}& 1  \ar[d]^{\ul{\eta}}\\
KFR' \ar[r]_{Kq} &  K\ul{K},}$$
 while the counit $\ul{\ve}$  is the unique natural
transformation $\ul{K} K\to 1 $ making the diagram
$$ \xymatrix @R=.4in{FR=FR'K \ar[rd]_-{\ve} \ar[r]^-{qK}& \ul{K} K \ar[d]^{\ul{\ve}}\\ &  1}$$
commute. Precomposing the image of the last square under $R'$ with $\eta R'$ and using the fact that $R'\gamma_{_K}$ is a monad morphism $\bT' \to \bT$,
we get the   commutative diagram
$$ \xymatrix @R=.4in{R' \ar[rd]_{\eta R'}\ar[r]^-{\eta' R'}&R'F'R'\ar[d]^-{R'\gamma_{_K}\! R'} \ar[rr]^-{R'\ve'}&& R'  \ar[d]^{R'\ul{\eta}}\\
&RFR'=R'KFR' \ar[rr]_{R'Kq=Rq} & & R'K\ul{K}.}$$
Since $R'\ve' \cdot \eta' R'=1$ by one of the triangular identities for
$F' \dashv R'$,  one gets
\begin{equation}\label{unit}
R'\ul{\eta}=Rq \cdot \eta R' .
\end{equation}

\begin{thm}\label{adj-comon}{\bf Comonads and adjunctions.} \em
Again let $\eta, \ve : F \dashv R: \A \to \B$ and $\eta', \ve' : F' \dashv R': \A \to \C$
be adjunctions
and now consider the corresponding comonads $\bG$ and $\bG'$. Let
$$\xymatrix{\B\ar[r]^K \ar[dr]_{F} & \C \ar[d]^{F'} \\
 & \A } $$
be a commutative diagram
 and
define $$\gamma_K: KR \xr{\eta' KR}R'F'KR=R'FR \xr{R'\varepsilon}R'.$$ Then $t_K=F'\gamma_K$ is a comonad morphism
$\bG \to \bG'$.
\end{thm}
\begin{theorem}\label{Dubuc2} In the situation described in \ref{adj-comon}, suppose that $F'$ is precomonadic.
Then $K$ has a right adjoint $\ol{K}$ if and only if the following equaliser exists in $[\C, \B]$,
$$\xymatrix@!=2.6pc{
\ol{K} \ar[r]^{\iota} & RF' \ar[rrr]^{RF'\eta'} \ar[dr]|{\eta RF'}&& & RF'R'F'\\
           & & RFRF'=RF'KRF' \ar[rru]|{Rt_{_K}F'=RF'\gamma_{_K} F'} & & . } $$
\end{theorem}
When this equaliser exists,  the unit $\ol{\eta}$ of the adjunction $K \dashv \ol{K}$ is the unique natural
transformation $1 \to \ol{K}K$ yielding commutativity of the diagram
$$\xymatrix @R=.5in{1 \ar[rd]_-{\eta} \ar[r]^-{\ol{\eta}}& \ol{K} K
\ar[d]^{\iota K}\\ &RF=RF'K  , }$$ while the counit $\ol{\ve}$ of the adjunction $K \dashv \ol{K}$ is the unique natural
transformation $K\ol{K} \to 1 $ making the diagram
$$ \xymatrix @R=.4in{K\ol{K}\ar[d]_-{\ol{\ve}} \ar[r]^-{K\iota}& KRF'  \ar[d]^{\gamma_{_K}\! F'}\\
1 \ar[r]_{\eta'} &  R'F'}$$
commute. Moreover, one has
\begin{equation}\label{counit.0}
F'\ol{\ve}=\ve F' \cdot F\iota .
\end{equation}

\begin{thm} \label{base}{\bf The restriction-  and change-of-base functors.}  \em Any morphism
$$\iota: \bS=(S, m^\bS, e^\bS) \to  \bT=(T, m^\bT, e^\bT)$$ of monads on $\A$ (that is, a natural transformation
$\iota : S \to T$ such that $\iota \cdot e^\bS=e^\bT$ and $\iota \cdot m^\bS=m^\bT \cdot (\iota \iota)$) gives
rise (see \cite{Be}) to the functor $$\iota^*: \A_{\bT} \to \A_\bS,\quad (a, h) \mapsto (a, h \cdot \iota_a),$$
called the \emph{restriction-of-base functor}. It is clear that $\iota^*$ makes the diagram
$$ \xymatrix{\A_\bT \ar[rd]_{U_\bT} \ar[r]^{\iota^*}& \A_\bS \ar[d]^{U_\bS}\\ &  \A}$$ commute.
Since the forgetful functor $U_\bS:\A_\bS \to \A$ is clearly monadic, it follows from the 
Theorem \ref{Dubuc1}
that $\iota^*$ has a left adjoint $\iota_! :\A_\bS \to \A_{\bT}$ if and only if the pair of natural transformations
\begin{equation}\label{base.0}
\xymatrix@!=2.9pc{\phi_\bT S U_\bS=\phi_\bT U_\bS \phi_\bS U_\bS \ar@{->}@<0.7ex>[rr]^-{\phi_\bT U_\bS \,\ve_{_\bS} }
\ar@{->}@<-0.7ex>[rr]_-{\varrho U_\bS}&& \phi_\bT U_\bS \ar@{-->}[r]^q & \iota_!\,,}
\end{equation} where  $\varrho:\phi_\bT S \to \phi_\bT$ is the composite $\phi_\bT S \xr{\phi_\bT\iota} \phi_\bT T
=\phi_\bT U_\bT \phi_\bT \xr{\ve_\bT \phi_\bT} \phi_\bT$, has a coequaliser $q:\phi_\bT U_\bS \to \iota_!$ in $[\A_\bS, \A_\bT]$.
This $\iota_!: \A_\bS \to \A_\bT$, when it exists, is called the \emph{change-of-base functor}. Recalling
that for any $\bS$-algebra $(a,g)$, $U_\bS((\ve_{_\bS})_{(a,g)})=g$ and $(\ve_{_\bT} \phi_\bT U_\bS)_{(a,g)}={m^\bT}\!\!_a$,
one finds that $\iota_!$ sends an $\bS$-algebra $(a,g)$ to the object $\iota_!(a,g)$ in the coequaliser diagram in $\A_{\bT}$,
\begin{equation}\label{ext.under}
\xymatrix{
\phi_\bT S(a) \ar[rr]^{\phi_\bT (g)} \ar[dr]_{\phi_\bT (\iota_a)} & &
\phi_\bT(a) \ar[r]^{q_{(a,g)}}&\iota_!(a,g)\\
&\phi_\bT T(a) \ar[ru]_{{m^\bT}\!\!_a} & &.}
\end{equation}
\end{thm}

\begin{remark} \label{rightiso.3} \em (1) Since $(\phi_\bT,  \ve_\bT \phi_\bT \cdot \phi_\bT \iota)$ is a right $\bS$-module by
Section \ref{canonical module}(vi), the diagram
$$
\xymatrix{\phi_\bT SS  \ar@{->}@<0.5ex>[rr]^-{\phi_\bT m^\bS} \ar@{->}@<-0.5ex>[rr]_{\ve_\bT \phi_\bT S \cdot \,\phi_\bT \iota S}&&
\phi_\bT S   \ar[r]^-{\phi_\bT \iota}& \phi^\bT T\ar[r]^-{\ve_\bT \phi_\bT} & \phi_\bT }
$$ is a coequaliser in $[\A,\A_\bT]$.
Observing that the pair $\xymatrix{\phi_\bT SS  \ar@{->}@<0.5ex>[rr]^-{\phi_\bT m^\bS} \ar@{->}@<-0.5ex>
[rr]_{\ve_\bT \phi_\bT S \cdot\,\phi_\bT \iota S}&&\phi_\bT S}$ is just the pair
 $\xymatrix @C=.6in{\phi_\bT S U_\bS \phi_\bS \ar@{->}@<0.5ex>[r]^-{\phi_\bT U_\bS \,\ve_{_\bS} \phi_\bS}\ar@{->}@<-0.5ex>[r]_-{\varrho U_\bS \phi_\bS}& \phi_\bT U_\bS \phi_\bS }$,
it follows that $q\phi_\bS$ is 
$\phi_\bT S \xr{\phi_\bT \iota} \phi_\bT T \xr{\ve_\bT \phi_\bT}\phi_\bT$.

(2) It is easy to see that $U_\bS \,\ve_{_\bS}\iota^* \phi_\bT$ is the composite $S T \xr{\iota T}TT \xr{m^\bT} T$
and since
\begin{itemize}
  \item $U_\bS \iota^* \phi_\bT=U_\bT\phi_\bT=T$,
  \item $U_\bT \ve_\bT \phi_\bT=m^\bT$,
  \item $U_\bT \varrho =TS \xr{T \iota}  TT \xr{m^\bT} T,$
\end{itemize} the pair $$(U_\bT\phi_\bT U_\bS \,\ve_{_\bS}\iota^* \phi_\bT, U_\bT\varrho U_\bS \iota^* \phi_\bT)$$ is just the pair
$$(Tm^\bT \cdot T \iota T, m^\bT T\cdot T \iota T).$$

Thus, if the coequaliser diagram (\ref{base.0}) exists and if its image under the functor
$[\A_\bS,U_\bT]:[\A_\bS,\A_\bT]\to [\A_\bS,\A] $ is again a coequaliser,
we get a coequaliser
\[\xymatrix {T S T\ar@{->}@<0.5ex>[rr]^-{Tm^\bT \cdot T \iota T}
\ar@{->}@<-0.5ex>[rr]_-{m^\bT T\cdot T \iota T}& & TT \ar[rr]^{U_\bT q\iota^* \phi_\bT} && U_\bT\iota_! \iota^* \phi_\bT}.\]
But since $(T, m^\bT \cdot  T \iota )$ is a right $\bS$-module, while $(T, m^\bT \cdot  \iota T)$ is a left
$\bS$-module (see Section \ref{canonical module}), one concludes that $U_\bT\iota_! \iota^* \phi_\bT =T\!\otimes_\bS \!T $
and that $U_\bT q\iota^*\phi_\bT=\text{can}_\bS^{T,T}$.
\end{remark}

\begin{proposition} \label{rightiso.2} If $\iota: \bS\to  \bT$ is a morphism of $\A$-monads such that
the change-of-base functor $\iota_!: \A_{\bS} \to \A_\bT$ exists, then the diagram
$$\xymatrix{\A \ar[dr]_{\phi_\bT} \ar[r]^-{\phi_\bS}& \A_\bS \ar[d]^-{\iota_!}\\ & \A_\bT}$$ commutes
(up to isomorphism).
\end{proposition}
\begin{proof} Since $U_\bS \cdot \iota^*=U_\bT$ (see Section \ref{base}) and  $\phi_\bS$ (resp.
$\iota_!$) is left adjoint to $U_\bS$ (resp. $\iota^*$), the result follows by 
 uniqueness of
left adjoints.
\end{proof}

\section{Separable Frobenius monads}\label{Frob}

The crucial role of separable Frobenius functors (e.g. \cite{St}) in the theory of weak bimonads
was pointed out by Szlach\'anyi in \cite{Szl-Adj} and such functors
are used by B\"ohm et al. in \cite{BLS} as an integral part of their definition of
weak bimonads on monoidal categories. In this section we show that in our approach separable Frobenius monads $\bS$
are of interest since they imply the existence of the change-of-base functor for monad morphisms
$S\to T$.

\begin{thm}\label{frob-mon}{\bf Definition.} \em An $\A$-monad $\bS=(S,m^\bS,e^\bS)$ is called
\emph{Frobenius} if $S$ has an $\A$-comonad structure $\bG=(S, \delta^\bS, \ve^\bS)$ with commutative
 diagram
\begin{equation}\label{sep-mon-d}
\xymatrix @C=.3in  @R=.16in {SS \ar[rd]^{m^\bS} \ar[rr]^{\delta^\bS S} \ar[dd]_{S \delta^\bS} && SSS \ar[dd]^{S m^\bS}\\
&S \ar[rd]^{\delta^\bS} &\\
SSS \ar[rr]_{m^\bS S}&& SS .}
\end{equation}

Dually, a comonad $\bG=(S, \delta^\bS, \ve^\bS)$ is called \emph{Frobenius}
if there exists a monad structure $\bS=(S,m^\bS,e^\bS)$ such that Diagram (\ref{sep-mon-d}) commutes.
$\bS$ (resp. $\bG$) is called \emph{Frobenius separable} if, moreover, $m^\bS \cdot \delta^\bS=1$.
\end{thm}

\begin{proposition} \label{sep.frobenius} Let $\bS$ be a Frobenius  separable monad on a Cauchy complete category $\A$.
Then for any morphism  $\iota: \bS \to \bT$ of monads, the change-of-base functor $\iota_!: \A_\bS \to \A_\bT$
exists.
\end{proposition}
\begin{proof} We claim that, under our assumptions,  (\ref{base.0}) is a split pair, a splitting morphism being
the composite
$$\pi:\xymatrix @C=.3in{\phi_\bT U_\bS \ar[r]^{\phi_\bT e^\bS U_\bS}& \phi_\bT S U_\bS \ar[r]^-{\phi_\bT \delta^\bS U_\bS}
&\phi_\bT SS U_\bS \ar[r]^-{\varrho SU_\bS} &\phi_\bT S U_\bS\,.} $$ Indeed, that $\varrho \cdot \pi=1$ follows from the
commutativity of the diagram
$$\xymatrix @C=.6in @R=.5in{\phi_\bT U_\bS \ar@/_3pc/[rrrd]_-{1}\ar[r]^{\phi_\bT e^\bS U_\bS}& \phi_\bT S U_\bS \ar[r]^-{\phi_\bT \delta^\bS U_\bS}
\ar@{=}[rd]&\phi_\bT SS U_\bS \ar[d]^{\phi_\bT m^\bS U_\bS}\ar[r]^-{\varrho SU_\bS} &\phi_\bT S U_\bS \ar[d]^-{\varrho U_\bS}\\
&&\phi_\bT S U_\bS \ar[r]^-{\varrho U_\bS}& \phi_\bT U_\bS\,.} $$
Here the square and the curved region commute since
$(\phi_\bT, \varrho)$ is a right $\bS$-module by Section \ref{canonical module}(vii), while the triangle commutes by separability of the monad $\bS$.

Next, to show that $\phi_\bT U_\bS \ve_\bS \cdot \pi \cdot \varrho U_\bS=\phi_\bT U_\bS \ve_\bS \cdot\pi \cdot \phi_\bT U_\bS \ve_\bS$,
consider the diagram
$$\xymatrix @R=.4in @C=.4in{
\phi_\bT SU_\bS \ar[dd]_-{\phi_\bT \delta^\bS  U_\bS}\ar[rd]_(.6){\phi_\bT S e^\bS
U_\bS}\ar[r]^-{\varrho U_\bS} &\phi_\bT U_\bS \ar[r]^-{\phi_\bT e^\bS U_\bS} &\phi_\bT
SU_\bS \ar[r]^-{\phi_\bT \delta^\bS U_\bS} & \phi_\bT
SSU_\bS \ar[r]^-{\varrho SU_\bS}& \phi_\bT SU_\bS   \ar[d]^-{\phi_\bT U_\bS \ve_\bS}\\
&\phi_\bT SSU_\bS \ar[d]_-{\phi_\bT \delta^\bS S U_\bS} \ar[ru]_{\varrho SU_\bS}
\ar[r]_-{\phi_\bT S\delta^\bS U_\bS}&\phi_\bT SSSU_\bS \ar[r]_-{\phi_\bT m^\bS SU_\bS}
\ar[ru]_{\varrho S SU_\bS}&\phi_\bT SSU_\bS \ar[ru]_{\varrho SU_\bS} &\phi_\bT U_\bS \\
\phi_\bT SSU_\bS \ar[r]_{\phi_\bT SS e^\bS U_\bS}&\phi_\bT SSSU_\bS \ar@/_2pc/[rru]_-{\phi_\bT Sm^\bS U_\bS}&&& }
$$
in which the curved region commutes since $\bS$ is assumed to be Frobenius, the right-hand parallelogram commutes since
$\varrho$ is a morphism of right $\bS$-modules, while the
other regions commute by naturality of composition. Thus the whole diagram is commutative, implying
-- since  $m^\bS \cdot S e^\bS=1$ -- that $$\phi_\bT U_\bS \ve_\bS \cdot \pi \cdot \varrho U_\bS=\phi_\bT U_\bS \ve_\bS
\cdot \varrho S U_\bS \cdot  \phi_\bT \delta^\bS  U_\bS.$$ In a similar manner one proves that
$$\phi_\bT U_\bS \ve_\bS \cdot\pi \cdot \phi_\bT U_\bS \ve_\bS=\phi_\bT U_\bS \ve_\bS \cdot \varrho S U_\bS \cdot  \phi_\bT
\delta^\bS  U_\bS .$$ So $\phi_\bT U_\bS \ve_\bS \cdot\pi \cdot \varrho U_\bS=\phi_\bT U_\bS \ve_\bS
\cdot\pi \cdot \phi_\bT U_\bS \ve_\bS.$ Therefore, the pair (\ref{base.0}) splits by the morphism $\pi$. Since $\A$ is assumed
to be Cauchy complete, $\A_\bT$ (and hence also the functor category $[\A_\bS, \A_\bT]$) is also Cauchy complete (see Proposition
\ref{idem}). It then follows that the pair (\ref{base.0}) admits a (split) coequaliser. Thus the extension-of-base functor
$\iota_!: \A_\bS \to \A_\bT$ exists.
\end{proof}

Dually, we have:

\begin{proposition} \label{sep.frobenius1}Let $\bF$ is a Frobenius  separable comonad on a Cauchy complete category $\A$.
Then for any morphism  $f:\bF \to \bG$ of comonads, the change-of-cobase functor $f_!: \A^\bG \to \A^\bF$
exists.
\end{proposition}

\section{Comparison functors}\label{Com}

Given a comonad $\bG$ on $\A$ and a category $\B$, one has the induced comonads $[\B, \bG]$ on $[\B, \A]$
and $[\bG, \B]$ on $[\A, \B]$.

\begin{thm}\label{com.adj.func.}{\bf Comodules and adjoint functors.} \em Consider a comonad
$\bG=(G,\delta,\varepsilon)$ on $\A$ and an adjunction $\eta, \sigma :F \dashv R:\B \to \A$.

There exist bijective correspondences (see \cite{G}) between:
\begin{itemize}
  \item functors $K : \B \to \A^\bG$ with $U^GK=F$;
 \item left $\bG$-comodule structures $\theta :F \to GF$ on $F$ (i.e., 
 $(F,\theta)\in[\B, \A]^{[\B, \bG]}$);
 \item comonad morphisms from the comonad generated by 
       $F \dashv R$ to the comonad $\bG$;
 \item right $\bG$-comodule structures $\vartheta :R \to RG$ on $R$ (i.e., 
 $(R,\vartheta)\in[\A, \B]^{[\bG, \B]}$).
\end{itemize}
\medskip

 These bijections are constructed as follows. If $U^\bG K=F$, then  $K(b)=(F(b), \theta_{b})$
for some morphism $\alpha_{b}: F(b) \to GF(b)$, and the collection $\{\theta_b,\, b \in\B\}$
constitutes a natural transformation $\theta:F \to GF$
making $F$ a left $\bG$-comodule.

Conversely, if $(F,\theta)\in [\B, \A]^{[\B, \bG]}$,
then $K: \B \to \A^\bG$ is defined by $K(b)=(F(b), \theta_b)$.

Next, for any left $\bG$-comodule structure $\theta :F \to GF$, the composite
$$t_K:FR \xr {\theta R} GFR \xr{G \sigma }  G $$
is a morphism from the comonad generated by 
$F \dashv R$ to the comonad
$\bG$. Then the corresponding right $\bG$-comodule structure $\vartheta :R \to RG$ on $R$ is 
$R \xr{ \eta R}RFR \xr{R t_K}RG .$

Conversely, for $(R, \vartheta) \in [\A,\B]^{[\bG, \B]}$,
 the corresponding comonad morphism $t_K:FR \to \bG$ is the composite
$$FR \xr {F \vartheta} FRG \xr {\sigma G} G,$$ while the corresponding left
$\bG$-comodule structure $\theta:F \to GF$ on $F$ is the composite $F \xr{F\eta} FRF \xr{t_K F}GF.$
\end{thm}

\begin{theorem} \label{fun.th.} Let $\bG=(G,\delta,\ve)$ be a comonad on $\A$
and $\eta, \sigma :F \dashv R:\B \to \A$ an adjunction. For a functor $K : \B \to \A^\bG$ with $U^\bG K=F$,
the following are equivalent:
\begin{itemize}
  \item [(a)] $K$ is an equivalence of categories;
  \item [(b)] $F$ is comonadic and the 
     comonad morphism ${t_K: FR \to \bG}$   is an isomorphism;
  \item [(c)] $F$ is comonadic and the composite $$\gamma_K : KR \xr{\eta^\bG KR} \phi^\bG U^\bG KR=\phi^\bG
  FR \xr{\phi^\bG \sigma} \phi^\bG$$ is an isomorphism.
\end{itemize}
\end{theorem}
\begin{proof} (a) and (b) are equivalent by \cite[Theorem 4.4]{M};   (b) and (c) are equivalent
since $U^\bG \gamma_K =t_K$ by \cite{D} and $U^\bG$ reflects isomorphisms.
\end{proof}

\begin{thm} \label{rght.adj}{\bf Right adjoint of $K$.} \em
Now fix a functor $K : \B \to \A^\bG$ with commutative diagram
\begin{equation}\label{triangle}
\xymatrix{ \B  \ar[r]^-{K}\ar[rd]_{F} & \A\!^\bG \ar[d]^{U^\bG}\\
& \A .}
\end{equation} Then $\gamma_K$ is the composite
$KR \xr{\eta^\bG KR} \phi^\bG U^\bG KR=\phi^\bG FR \xr{\phi^\bG \ve} \phi^\bG$ and using
the fact from Section \ref{adj-comon} that $U^\bG \gamma_K $ is just the comonad morphism $t_K : FR \to G$
induced by the triangle, an easy calculation shows that
$$\beta U^\bG=RU^\bG\gamma_K U^\bG \cdot \eta R U^\bG,$$
where $\beta : R \to RG$ is the right $\bG$-module structure on $R$ corresponding to
the triangle (\ref{triangle}). Thus, when the right adjoint $\ol{K}$ of $K$ exists, it
is determined by the equaliser diagram
\begin{equation}\label{rght.adj.sp}\xymatrix{
\ol{K} \ar[r]^{\iota\quad} & R U^\bG
\ar@{->}@<0.5ex>[rr]^-{RU^\bG \eta^\bG} \ar@ {->}@<-0.5ex> [rr]_-{\beta
U^\bG }&& RGU^\bG=RU^\bG \phi^\bG U^\bG.}\end{equation}

\noindent It is easy to see that for any $(a, \theta)\in \A^\bG$, the $(a, \theta)$-component of
(\ref{rght.adj.sp}) is the equaliser diagram
\begin{equation}\label{rght.adj.com}\xymatrix{\ol{K}(a, \theta)\ar[r]^-{\iota_{(a,\theta)}} &R(a) \ar@{->}@<0.5ex>[r]^-{R(\theta)} \ar@
{->}@<-0.5ex> [r]_-{\beta_a} & RG(a)\,.}\end{equation}

It is easy to verify, using (\ref{counit.0}) that
if $\ol{\sigma}$ is the counit of the adjunction $K\dashv \ol{K}$,
then for any $(a, \theta)\in \A\!^\bG$, one has
\begin{equation}\label{counit}U^\bG(\ol{\sigma}_{(a, \,\theta)})=\sigma_{a} \cdot F(\iota_{(a,\,\theta)})\,.\end{equation}
\end{thm}

Suppose now that $\ol{K}$ exists, write $\bP$ for the monad on $\B$ generated by the adjunction
$K \dashv \ol{K}$, and consider the corresponding comparison functor $K_\bP :\A\!^\bG \to \B_\bP$.
Then $K_\bP(a, \theta)=(\ol{K}(a, \theta), \ol{K}(\ol{\sigma}_{(a, \,\theta)})$ for any
$(a, \theta) \in \A\!^\bG$. Moreover, $ K_\bP K=\phi_\bP \,\,\, \text{and}\,\,\, U_\bP K_\bP=\ol{K}$.
The situation may be pictured as
\begin{equation}\label{Diag}
\xymatrix@C=.6in {\B \ar@/_1pc/@{->}[rr]_{\phi_\bP}\ar@/_3pc/@{->}[rrrr]_{K}\ar[rrddd]_{F} & &
\B_{\bP}  \ar[ll]_{U_{\bP}}& & 
\A^{\bG} \ar@/_2pc/@{->}[llll]_{\ol{K}}  \ar[llddd]^{U^{\bG}} \ar[ll]_{K_\bP}\\\\
&&&&\\
 & &  \A & &}
\end{equation}

In order to proceed, we need the following (see \cite[Lemma 21.2.7]{S}).
\begin{proposition}\label{S}
Let $\eta, \sigma :F \dashv R : \A \to \C $ and
$\eta ' , \sigma' :F' \dashv R' : \B \to \C $ be adjunctions with corresponding
monads $\bT$ and $\bT\,'$, respectively,
and let
$$ \xymatrix{\A  \ar[r]^{K}& \B \\
 &  \C \ar[lu]^{F}\ar[u]_{F'}}
$$
be a commutative diagram of categories and functors.
Then the composite
$$\xymatrix{T=RF \ar[r]^-{\eta' RF}& R'F'RF=R'KFRF \ar[r]^-{R'K
\sigma F}& R'KF=R'F'=T'}$$ is a monad morphism $\bT \to \bT\,'$.
\end{proposition}

Suppose again that $\ol{K}$ exists and consider the natural transformation $\iota: P \to RF$, where
$\iota_b=\iota_{K(b)}$ for all $b\in \B$.

\begin{proposition}\label{monad.mor.} $\iota: P \to RF$ is a monad morphism from the monad $\bP$ to the
monad generated by the adjunction $F \dashv R$.
\end{proposition}

\begin{proof} Applying Proposition \ref{S} to the diagram
$$\xymatrix@R=.2in{ \A\!^{\bG} \ar[rr]^{U^{\bG}} \ar[ddr]^{\ol{K}}&&
\A \ar[ddl]_{R}\\\\
&\B \ar@/^/@<+1.5ex>[uul]^{K} \ar@/_/@<-1.5ex>[uur]_{F}&}$$
in which $U^{\bG} K=F$, gives that the natural transformation
$$P=\ol{K}K \xr{\eta \ol{K}K} RF \ol{K}K=RU^{\bG}K \ol{K}K
\xr{R U^{\bG} \ol{\sigma}K} R U^{\bG}K=RF$$
is a monad morphism from the monad $\bP$ to the
monad generated by the adjunction $F \dashv R$.
Since for any $(a, \theta)\in \A\!^\bG$,
$U^\bG(\ol{\sigma}_{(a, \,\theta)})=\sigma_{a} \cdot F(\iota_{(a,\,\theta)})$ by (\ref{counit}),
it follows that, for each $b \in \B$, the $b$-component of the above natural transformation is the composite
$$P(b) \xr{\eta_{P(b)}} UFP(b) \xr {UF(\iota_{b})}UFUF(b) \xr {U \sigma_{F(b)}}UF(b),$$
which is easily verified to be just the morphism $\iota_b: P(b) \to UF(b).$ This completes
the proof.
\end{proof}

We are mainly interested in the case where the functor $F$ is monadic. So, our standard situation of interest,
and our standard notation, will henceforth be as follows. We consider a monad
 $\bT=(T,m^\bT,e^\bT)$ on $\A$, a comonad
$\bG$ on $\A_\bT$ and an adjunction $\ol{\eta }\,, \ol{\sigma}:K \dashv \ol{K} : (\A_\bT)^\bG \to \A$,
where $K:\A \to (\A_\bT)^\bG$ is a functor with $U_\bG K=\phi_\bT.$
Write $\bP=(P,m^\bP,e^\bP)$ for the monad on $\A$ generated
by the adjunction $K \dashv \ol{K}$ and write $\iota: P \to T$ for the induced morphism of monads. This
is pictured   in the   diagram
\begin{equation}\label{Diag1}
\xymatrix@C=.6in {\A \ar@/_1pc/@{->}[rr]_{\phi_\bP}\ar@/_3pc/@{->}[rrrr]_{K}\ar[rrddd]_{\phi_\bT} & &
\A_{\bP}  \ar[ll]_{U_{\bP}}\ar@/_1pc/@{-->}[rr]_(0.35){L_\bP}& &
(\A_\bT)^{\bG} \ar@/_2pc/@{->}[llll]_{\ol{K}}  \ar[llddd]^{U^{\bG}} \ar[ll]_{K_\bP}\\\\
&&&&\\
 & &  \A_\bT & &,}
\end{equation}
in which $K_\bP:(\A_\bT)^{\bG} \to \A_\bP$ is the Eilenberg-Moore comparison functor for the monad $\bP$,
and thus $ K_\bP K=\phi_\bP \,\,\, \text{and}\,\,\, U_\bP K_\bP=\ol{K}$.

\begin{proposition}\label{left.adjoint} In the situation above, the functor $K_\bP:(\A_\bT)^{\bG} \to \A_{\bP}$
admits a left adjoint $L_\bP:\A_{\bP} \to (\A_\bT)^{\bG}$ if and only if the restriction-of-base functor
$\iota^*: \A_\bT \to \A_\bP$ admits a left adjoint, i.e., the change-of-base functor $\iota_! : \A_\bP \to \A_\bT$ exists.
Moreover, when this is the case, $\iota_!=U_\bG L_\bP$.
\end{proposition}
\begin{proof} 
 According to Section \ref{base}, $\iota_! : \A_\bP \to \A_\bT$ exists if and only if
  for each $(a, g) \in \A_{\bP}$, the pair of morphisms $(\phi_\bT(g), m^\bT_a \cdot \phi_\bT(\iota_a))$
  has a coequaliser in $\A_\bT$, while by Proposition \ref{Beck}(1), $L_\bP:\A_{\bP} \to (\A_\bT)^{\bG}$ exists
  if and only if the pair of morphisms $(K(g), \ol{\sigma}_{K(a)})$
  has a coequaliser in $(\A_\bT)^{\bG}$.

 Since the functor $U^\bG : (\A_\bT)^{\bG} \to \A_\bT$ preserves and creates coequalisers,
it suffices to show that the image of the pair $(K(g), \ol{\sigma}_{K(a)})$ under $U^\bG$ is
just the pair $(\phi_\bT(g), m^\bT_a \cdot \phi_\bT(\iota_a))$. That $U^\bG K(g)=\phi_\bT(g)$ follows from
the equality $U^\bG K=\phi_\bT$. Next, by (\ref{counit}), $U^\bG(\ol{\sigma}_{K(a)})=
(\ve_\bT)_{U^\bG K(a)} \cdot \phi_\bT (\iota_{K(a)})$. But since $(\ve_\bT)_{U^\bG K(a)}=
(\ve_\bT)_{\phi_\bT(a)}=m^\bT_a$, we see that $U^\bG(\ol{\sigma}_{K(a)})=m^\bT_a \cdot \phi_\bT(\iota_a)$.
Hence $$U^\bG(K(g),\ol{\sigma}_{K(a)})=(\phi_\bT(g), m^\bT_a \cdot \phi_\bT(\iota_a))$$ and thus the result follows.
\end{proof}

Now assume that the change-of-base functor $\iota_! : \A_\bP \to \A_\bT$ exists, that is,
 $K_\bP:(\A_\bT)^{\bG} \to \A_{\bP}$ admits a left adjoint $L_\bP:\A_{\bP}
\to (\A_\bT)^{\bG}$. Thus, for any $(a,g) \in \A_\bP$, $\iota_!(a,g)$ is given by be the coequaliser
$$
\xymatrix{TP(a) \ar@/^2pc/@{->}[rr]^{T(g)}
\ar[r]_-{T(\iota_a)} & TT(a) \ar[r]_-{m_a} & T(a) \ar[r]^{q_{(a,\,g)}}& \iota_!(a,g)}.$$

Since $\iota_!=U_\bG L_\bP$ by Proposition \ref{left.adjoint} and $\iota_! \cdot \phi_\bP=\phi_\bT$ by
Proposition \ref{rightiso.2}, both triangles in the diagram
\begin{equation}\label{diagram.2}
\xymatrix@C=.6in  @R=.2in{\A \ar[r]_{\phi_\bP}\ar@/^1.8pc/@{->}[rr]^{K}\ar[rdd]_{\phi_\bT}  &
\A_{\bP}  \ar[r]_{L_\bP} \ar[dd]_{\iota_!}&
(\A_\bT)^{\bG}  \ar[ldd]^{U^{\bG}}\\ &&\\  &  \A_\bT  &}
\end{equation} commute. Write $\bG_1$ (respectively $\bG_2$) for the $\A_\bT$-comonad generated by the adjunction $\phi_\bT \dashv U_\bT$
(respectively $\iota_! \dashv \iota^*$), and consider the related comonad morphism $t_{\phi_\bP}:\bG_1 \to \bG_2$
(respectively $t_{L_\bP}:\bG_2 \to \bG$) corresponding to the left (respectively right) triangle in the above diagram (see
Sections \ref{adj-comon} and \ref{com.adj.func.}). Since $U_\bP K_\bP=\ol{K}$ and  $\phi_\bP$ (respectively $L_\bP$) has a right adjoint
$U_\bP$ (respectively $K_\bP$), it follows -- by  uniqueness of right adjoints -- that $L_\bP \cdot \phi_\bP=K.$ Thus we may
apply \cite[Proposition 1.21]{MW} to obtain the equality
\begin{equation}\label{composition}
t_K=t_{L_{_{\bP}}} \cdot t_{\phi_{_\bP}}.
\end{equation}

Recall from Section \ref{base} that $\iota_!$ can be obtained as the coequaliser of  Diagram (\ref{base.0}).

\begin{proposition}\label{can.epi} If $K_\bP:(\A_\bT)^{\bG} \to \A_{\bP}$ admits a left adjoint, then \,
 $t_{\phi_{_\bP}}=q\iota^*$.
\end{proposition}
\begin{proof} Applying the results of Section \ref{adj-comon} to the left triangle in  Diagram (\ref{diagram.2}) gives that
$t_{\phi_\bP}=\iota_! \gamma$, where $\gamma$ is the composite
$$\phi_\bP U_\bT \xr{\ul{\eta} \,\phi_\bP U_\bT} \iota^*\iota_!\phi_\bP U_\bT=\iota^* \phi_\bT U_\bT\xr{\iota^* \ve_{_\bT}}\iota^*.$$
Here $\ul{\eta}$ is the unit of the adjunction $\iota_! \dashv \iota^*$, which (as a result of Theorem \ref{Dubuc2} applied to
the commutative diagram $U_\bP \cdot \iota^*=U_\bT$) is a unique natural transformation making the diagram
$$
\xymatrix @R=.4in @C=.4in{\phi_\bP U_\bP \ar[rr]^-{\ve_{_\bP}} \ar[d]_{\gamma' U_\bP}  && 1 \ar[d]^{\ul{\eta}}\\
\iota^* \phi_\bT U_\bP \ar[rr]_-{\iota^* q} &&\iota^*\iota_!}
$$  commute. Here $\gamma'$ is the composite
$$\phi_\bP \xr{\phi_\bP \, \eta_\bT} \phi_\bP U_\bT \phi_\bT= \phi_\bP U_\bP \iota^*\phi_\bT \xr{\ve_\bP \,\iota^*\phi_\bT} \iota^*\phi_\bT.$$
Since $U_\bP \gamma'=\iota$ and $\iota \cdot e^\bP=e^\bT$, the diagram
$$
\xymatrix @R=.6in @C=.5in{U_\bP \ar[r]^-{e^\bP U_\bP} \ar[rd]_{e^\bT U_\bP}& PU_\bP=U_\bP\phi_\bP U_\bP \ar[rr]^-{ U_\bP\,\ve_{_\bP}} \ar[d]^{\iota U_\bP=U_\bP\gamma' U_\bP}  && U_\bP \ar[d]^{U_\bP\ul{\eta}}\\
&T U_\bP=U_\bP\iota^* \phi_\bT U_\bP \ar[rr]_-{U_\bT q=U_\bP\iota^* q} && U_\bP\iota^*\iota_!}
$$  commute. Since $U_\bP\ve\!_{_\bP}\cdot e^\bP U_\bP=1$ by one of the triangular identities
for the adjunction $\phi_\bP \dashv U_\bP$, it follows that $U_\bP\ul{\eta}$ is the composite
$$U_\bP \xr{e^\bT U_\bP} TU_\bP=U_\bP\iota^* \phi_\bT U_\bP \xr{U_\bP\,\iota\!^* q} U_\bP\iota^*\iota_!.$$
In particular, $U_\bP \ul{\eta}\phi_\bP$ is the composite
$$P \xr{e^\bT P} TP=U_\bP\iota^* \phi_\bT U_\bP \phi_\bP \xr{U_\bP \,\iota\!^* q\,\phi_\bP} U_\bP\iota^*\iota_! \phi_\bP=U_\bP\iota^* \phi_\bT=U_\bT\phi_\bT.$$

\noindent Since, by Remark \ref{rightiso.3}(1),  $q \phi_\bP$ is the composite
$$\phi_\bT P \xr{\phi_\bT \iota} \phi_\bT T \xr{\ve_\bT \phi_\bT}\phi_\bT,$$ $U_\bP \ul{\eta}\phi_\bP$ is the composite
$$P \xr{e^\bT P}TP \xr{T\iota} TT \xr{m^\bT} T.$$
Since, by functoriality of composition,
$T\iota \cdot e^\bT P=e^\bT T \cdot\iota$ and since $m^\bT \cdot e^\bT=1$, one concludes that
$U_\bP\ul{\eta}{\phi_\bP}=\iota.$

Now since for any $(a,h)\in \A_\bT$, $(\ve_{_\bT})_{(a,\,h)}=h$,
one concludes that $\gamma_{(a,h)}$ is the composite $P(a) \xr{\iota_a} T(a) \xr{h} a.$ Now, since by Remark \ref{rightiso.3},
$$
\xymatrix{TPP(a)  \ar@{->}@<0.5ex>[rr]^-{T(m^\bP_a)}
\ar@{->}@<-0.5ex>[rr]_{m^\bT_a \cdot \,T(\iota_{P(a)})}&& TP(a)  \ar[r]^-{T(\iota_a)}& TT(a)
\ar[r]^-{m^\bT_a} & T(a)}
$$ is the coequaliser defining $\iota_!(P(a), m^\bP_a)=\iota_!(\phi_\bP(a))$,
it follows that $\iota_! (\gamma_{(a,\,h)})=(t_{\phi_\bP})_{(a,\,h)}$ is the unique morphism making the diagram
$$\xymatrix @C=.4in @R=.2in {TP(a) \ar[r]^-{T(\iota_a)} \ar[d]_{T(\iota_a)}& TT(a) \ar[r]^-{m^\bT_a}& T(a) \ar[dd]^{\iota_! (\gamma_{(a,h)})}\\
TT(a) \ar[d]_{T(h)} &&\\
T(a) \ar[rr]_-{q_{\iota\!^*(a,h)}}&&\iota_!(\iota^*(a,h))}
$$  commute. Since $\iota_a \cdot e^\bP_a=e^\bT_a$ and  $m^\bT_a \cdot T(e^\bT_a)=1=h \cdot T(e^\bT_a)$,
it follows from the commutativity of the above diagram that
$(t_{\phi_\bP})_{(a,\,h)}=\iota_! (\gamma_{(a,h)})=q_{\iota\!^*(a,\,h)}$, as claimed.
\end{proof}

\section{Weak entwinings}\label{Entw}

Let $H$ be an endofunctor on any category $\A$, admitting both a monad $\ul{\bH}=(H, m, e)$  and a comonad
$\ol{\bH}=(H, \delta,\ve)$ structure, and define
\begin{equation}\label{fusion}
\begin{array}[t]{c}
\sigma  :HH \xr{\delta H} HHH\xr{Hm} HH ,\\[+1mm]
\ol{\sigma} :HH \xr{H\delta } HHH\xr{mH} HH .
\end{array}
\end{equation}

The class $\mathbf{Nat}(H,H)$ of all natural transformations from $H$ to itself
allows for the structure of a monoid by defining the (convolution) product of any two
$\varphi, \varphi'\in \mathbf{Nat}(H,H)$ as the composite
$\varphi * \varphi'=m\cdot \varphi  \varphi'\cdot \delta$.
  The identity for this product is $e \cdot \ve :H \to H$.
\smallskip

Recall that weak entwinings of tensor functors were defined by Caenepeel and De Groot
in \cite{CDG} and a more general theory was formulated by B\"ohm (e.g. \cite[Example 5.2]{Bo}).

\begin{thm} \label{w-entw}{\bf Weak monad comonad entwinings.} \em
For a natural transformation $\omega :HH \to HH$,
  define the natural transformations
$$\begin{array}{rl}
\xi:& H\xr{e H} HH \xr{\;\omega\;} HH \xr{\ve H} H,\\
\kappa : & HH\xra{e HH} HHH \xra{\omega H} HHH \xra{Hm} HH .
\end{array}
$$

$(\ul H,\ol H, \omega)$ is called a {\em weak entwining}
(from the monad  $\ul \bH$ to the comonad $\ol \bH$) provided
 \begin{equation}\label{ent}
\begin{array}[t]{lrl}\hspace{-7mm}
{\rm (i)} &  \omega\cdot mH = Hm\cdot \omega H\cdot H\omega,&
 \delta H\cdot \omega = H\omega \cdot \omega H \cdot H\delta, \\[+1mm]
\hspace{-7mm} {\rm (ii)} &
  \omega\cdot eH =H\xi\cdot \delta   ,& \ve H \cdot \omega = H\xi\cdot \delta,
\end{array}
\end{equation}
and is said to be {\em compatible} if
 \begin{equation}\label{compatible}
\delta \cdot m= Hm\cdot \omega H\cdot H\delta.
\end{equation}

It is easily checked that
\begin{equation} \label{c-diagrams}
\begin{array}[t]{cl}
 \kappa\cdot He = \omega\cdot eH,\;\, \ve H \cdot \kappa = m\cdot \xi H, &
  \mbox{always hold,} \\[+1mm]
 \xi*\xi=\xi,\quad  \kappa\cdot \kappa=\kappa, \quad
  \kappa\cdot \omega = \omega, & \mbox{follow by  (\ref{ent})(i)}, \\[+1mm]
\kappa\cdot \delta=\delta,\quad \kappa\cdot \sigma =\sigma, \quad \xi*1=1, &
  \mbox{follow by  (\ref{compatible})}.
\end{array}
 \end{equation}
\end{thm}


\begin{thm}\label{mix-bim}{\bf Mixed bimodules.} \em
We write $\A_{\,\ul\bH}^{\ol\bH}(\omega)$ for the category of mixed $H$-bimodules,
whose objects are triples $(a, h, \theta)$,
where $(a, h)\in \A_{\ul\bH}$, $(a, \theta) \in \A^{\ol\bH}$ with commutative diagram
$$
\xymatrix{ H(a) \ar[r]^{h} \ar[d]_{H(\theta)} & a \ar[r]^{\theta} & H(a)\\
HH(a) \ar[rr]^{\omega_a} && HH(a) , \ar[u]_{H(h)} }
$$
and whose morphisms  are those in $\A$ which are $\ul\bH$-module as well as $\ol\bH$-comodule morphisms.
\end{thm}

The following is a particular case of \cite[Proposition 5.7]{Bo}.

\begin{proposition} \label{Bohm}
Let $\bH=(\ul{\bH},\ol{\bH},\omega)$ be a    
weak entwining on $\A$. Then the composite
$$\Gamma: HU_{\ul{\bH}} \xr{eHU_{\ul{\bH}}} HHU_{\ul{\bH}}
 \xr{\omega U_{\ul{\bH}}} HHU_{\ul{\bH}}=HU_{\ul{\bH}}\phi_{\ul{\bH}}U_{\ul{\bH}}
 \xr{H U_{\ul{\bH}} \,\ve_{\ul{\bH}}} HU_{\ul{\bH}}$$ is an idempotent, and if
$$ \xymatrix{ HU_{\ul{\bH}} \ar@{->>}[rd]_{p}\ar[rr]^{\Gamma} && HU_{\ul{\bH}}\\
& G\; \ar@{>->}[ru]_{i} &}$$ is its splitting, then there is a comonad
$\bG=(\widetilde{G}, \widetilde{\delta},\widetilde{\ve})$ on $\A_{\ul{\bH}}$, whose functor part 
takes an arbitrary $(a,h)\in\A_{\ul{\bH}}$ to
$$(G(a,h),\, p_{(a,h)} \cdot H(h) \cdot \omega_a \cdot H(i_{(a,h}): HG(a,h) \to G(a,h)),$$
and whose comultiplication $\widetilde{\delta}$ and counit $\widetilde{\ve}$ evaluated at
$(a,h)$ are the composites, respectively,
$$G(a,h)\xr{i_{(a,h)}} H(a) \xr{\delta_a} HH(a)\xr{H(p_{(a,h)})}HG(a,h) \xr{p_{G(a,h)}}GG(a,h),$$
$$G(a,h) \xr{i_{(a,h)}} H(a) \xr{h} a .$$
\end{proposition}

We call $\bG$ the \emph{comonad induced by} 
$\bH=(\ul{\bH},\ol{\bH},\omega)$. Obviously,
$U_{\ul{\bH}} \widetilde{G} =G$.

\begin{theorem}\label{comparison} Let $\bH=(\ul{\bH},\ol{\bH},\omega)$ be a  
weak entwining on a Cauchy complete category
$\A$ and $\bG$ the induced comonad  on $\A_{\ul{\bH}}$. Then
there is an isomorphism of categories
$$\Phi:\Aa \to (\A_{\ul{\bH}})^\bG,\quad
(a, h, \theta)\; \mapsto \;((a, h), p_{(a,h)}\cdot \theta),$$
with the inverse 
given by $\Phi^{-1}((a,h), \zeta)
=(a,h, i_{(a,h)}\cdot \zeta).$
\end{theorem}
\begin{proof}Since $p \,i =1$, it is clear that $\Phi\Phi^{-1}=1$. To show that $\Phi^{-1}\Phi=1$, consider
an arbitrary object $(a,h,\theta) \in \Aa.$  In the diagram
$$
\xymatrix @R=.4in @R=.5in{a \ar[d]_{\theta} \ar[r]^{e_{a}}& H(a) \ar[r]^{h} \ar[d]_{H (\theta)} & a \ar[rd]^{\theta}&\\
H(a) \ar@<+2pt> `d[r]`[rrr]_{\Gamma_{(a,h)}=i_{(a,h)} \cdot p_{(a,h)}} [rrr]\ar[r]^{e_{H(a)}} &HH(a) \ar[r]^{\omega_{a}}
& HH(a) \ar[r]^{H(h)}&H(a)}$$
 the square commutes by naturality of $e$, while the trapezium commutes since
$(a,h,\theta) \in \Aa.$ Since $h \cdot e=1$, this means
$$\theta=\Gamma_{(a.h)}\cdot \theta=i_{(a,h)} \cdot p_{(a,h)}\cdot \theta . $$
Thus $\Phi^{-1}\Phi (a,h,\theta)=(a,h,\theta)$,
that is, $\Phi^{-1}\Phi=1$.
\end{proof}

Again by \cite[Proposition 5.7]{Bo}, we get as counterpart of
Proposition \ref{Bohm}:

\begin{proposition} \label{Bohm1} Let $\bH=(\ul{\bH},\ol{\bH},\omega)$ be a
weak entwining 
on  $\A$.
Then
$$\Gamma': HU^{\ol{\bH}} \xr{HU_{\ol{\bH}}\,\eta^{\ol{\bH}}} HU^{\ol{\bH}}\phi^{\ol{\bH}}U^{\ol{\bH}}=HHU^{\ol{\bH}}
\xr{\omega U^{\ol{\bH}}} HHU^{\ol{\bH}}\xr{\ve H U^{\ol{\bH}}} HU^{\ol{\bH}}$$ is an idempotent, and if
$$ \xymatrix{ HU^{\ol{\bH}} \ar@{->>}[rd]_{p'}\ar[rr]^{\Gamma'} && HU^{\ol{\bH}}\\
& T\; \ar@{>->}[ru]_{i'} &}$$ is its splitting, then there is a comonad $\bT=(\widetilde{T}, \widetilde{m},
\widetilde{e})$ on $\A^{\ol{\bH}}$, whose functor part 
takes an arbitrary $(a,\theta)\in\A^{\ol{\bH}}$ to
$$(T(a,\theta), H(p_{(a,\theta)}') \cdot \omega_a \cdot H(\theta)\cdot i'_{(a,\theta)}: T(a,\theta) \to HT(a,\theta)),$$
and whose multiplication $\widetilde{m}$ and unit $\widetilde{e}$, evaluated at an
$\ol{\bH}$-comodule $(a,\theta)$, are the composites, respectively,
$$\begin{array}{c}
TT(a,\,\theta)\xr{i'_{T(a,\theta)}} HT(a,\theta) \xr{H(i'_{(a,\theta)})} HH(a)\xr{m_a}H(a) \xr{p'_{G(a,\theta)}}T(a,\theta),\\[+1mm]

a \xr{\;\theta\;} H(a)\xr{p'_{(a,\theta)}} T{(a,\theta)}.
\end{array}
$$
\end{proposition}
We call $\bT$ the \emph{monad induced by} 
 $\bH=(\ul{\bH},\ol{\bH},\omega)$. Obviously,
$U^{\ol{\bH}} \widetilde{T}=T$.

\begin{theorem}\label{comparison1} Let $\bH=(\ul{\bH},\ol{\bH},\omega)$ be a  
weak entwining on a Cauchy complete category
$\A$ and $\bT$ the induced monad on $\A^{\ol{\bH}}$. Then
there is an isomorphism of categories
$$\Phi':\Aa \to (\A^{\ol{\bH}})_\bT, \quad
 (a, h, \theta)\;\mapsto\; ((a,\theta), h \cdot i'_{(a,\theta)}),$$
with the inverse 
given by $(\Phi')^{-1}((a,\theta), g)
=(a, g \cdot i'_{(a,h)}, \theta).$
\end{theorem}

\begin{thm}\label{comp-func}{\bf Comparison functors.} \em
Let $\bH=(\ul{\bH},\ol{\bH},\omega)$ be a 
compatible weak entwining on $\A$.
By  (\ref{compatible}), there is a functor (e.g. \cite[Lemma 5.1]{LMW})
$$K_{\omega}:\A \to \Aa,\quad a \; \mapsto \; (H(a),m_a, \delta_a).$$
 Precomposing $K_{\omega}$ with $\Phi$ and $\Phi'$ gives functors
 $$\begin{array}{ll}
K: \A \to (\A_{\,\ul{\bH}})^\bG, & a \;\mapsto \;
((H(a),m_a), \, p_{(H(a),\,m_a)}\!\cdot \delta_a),\\[+2mm]
K': \A \to (\A^{\ol{\bH}})_\bT,& a\;\mapsto\;
((H(a),\delta_a),\, m_a\!\cdot i'_{(H(a),\, m_a)}),
\end{array}$$
leading to commutative diagrams
\begin{equation}\label{triangle1}
\xymatrix{ 
\A \ar[r]^{K\quad} \ar[d]_{\phi_{\ul{\bH}}} &
\Ab  \ar[ld]^{U^{\bG}}\\
  \A_{\ul{\bH}} & , }\quad
\xymatrix{
\A \ar[r]^{K'\quad} \ar[d]_{\phi^{\ol{\bH}}} &
(\A^{\ol{\bH}})_\bT  \ar[ld]^{U_{\bT}}\\
  \A^{\ol{\bH}} & , } \quad
\xymatrix @C=.4in{
 \A \ar[r]^{K\;} \ar[d]_{K'} \ar[dr]_{K_\omega} & \Ab  \\
\Ac &\ar[l]^{\Phi'}   \Aa   \ar[u]_{\Phi} \, .}
\end{equation}
\end{thm}

We will use that the splitting of $\Gamma$ (from \ref{Bohm}) leads to a  splitting of $\kappa$,
\begin{equation}\label{olG}
\xymatrix{ 
 HH \ar@{->>}[rd]_{\ol{p}=p \,\phi_{\ul{\bH}}}\ar[rr]^{\kappa} && HH\\
& \ol{G}=G\phi_{\ul{\bH}} \ar@{>->}[ru]_{\ol{i}=i\phi_{\ul{\bH}}} &.}
\end{equation}

Symmetrically, the splitting of $\Gamma'$ (from \ref{Bohm1}) leads to a  splitting of $\kappa'$,
\begin{equation}\label{olT}
\xymatrix{ 
 HH \ar@{->>}[rd]_{\ol{p}\,'=p' \,\phi_{\ul{\bH}}}\ar[rr]^{\kappa'} && HH\\
& \ol{T}=T\phi_{\ul{\bH}} \ar@{>->}[ru]_{\ol{i}\,'=i'\phi_{\ul{\bH}}} &,}
\end{equation} where $\kappa'$ is the composite
$HH\xr{HHe}HHH \xr{H\omega}HHH \xr{mH}HH$.

\begin{proposition}\label{coaction}
In the situation described above, consider the comonad morphism
$t:\phi_{\ul{\bH}}U_{\ul{\bH}} \to \bG$ induced by the left triangle in (\ref{triangle1})
(see Section \ref{com.adj.func.}).
\begin{zlist}
\item  For any $\ul \bH$-module $(a,h)$, the $(a,h)$-component for $t$ is the composite
$$t_{(a,h)}:H(a) \xr{\;\delta_a\;}HH(a) \xr{H(h)} H(a) \xr{p_{(a,h)}} G(a,h).$$
\item
 For any $a \in \A$, the $\phi_{\ul{\bH}}(a)$-component for $t$,
 $$t_{\phi_{\ul{\bH}}(a)}:
HH(a)\xr{\delta_{H(a)}}HHH(a) \xr{Hm_a} HH(a) \xr{\;\ol p_{a}\;} \ol G(a),$$
 is the unique morphism leading to commutativity of the diagram
  $$\xymatrix @R=.4in @C=.4in{
HH(a) \ar[rd]_{\sigma_a} \ar[r]^-{t_{\phi_{\ul{\bH}}(a)}}& \ol{G}(a)\ar[d]^{\ol{i}_a}\\
  & HH(a).}$$
\end{zlist}
\end{proposition}
\begin{proof} (1) Since $K(a)=((H(a),m_a),p_{_{(H(a),\,m_a)}}\! \cdot \delta_a)$, the left $\bG$-comodule
structure $\alpha : \phi_{\ul{\bH}} \to \bG \phi_{\ul{\bH}}$ on $\phi_{\ul{\bH}}$
corresponding to the left triangle in (\ref{triangle1}), has for its $a$-component
 $\alpha_a =p_{_{(H(a),\,m_a)}} \cdot \delta_a$.
 It then follows from Section \ref{com.adj.func.} that for any
$(a,h) \in \A_{\ul{\bH}}$, the $(a,h)$-component $t_{(a,h)}$ is the composite
$$t_{(a,h)} : H(a) \xr{\;\delta_a\;}
HH(a) \xr{p_{_{(H(a),\,m_a)}}} G(H(a),m_a) \xr{G(h)} G(a,h),$$
which, by naturality of composition, is the same as
$$t_{(a,h)} :H(a) \xr{\delta_a}HH(a) \xr{H(h)} H(a) \xr{p_{(a,h)}} G(a,h).$$
Then, in particular, $t_{\phi_{\ul{\bH}}(a)}=p_{\phi_{\ul{\bH}}(a)}\cdot H(m_a)\cdot \delta_{H(a)}=\ol{p}_a \cdot \sigma_a$.

(2) Since
$\xymatrix{ \ol{G} \ar[r]^-{\ol{i}} & HH \ar@{->}@<0.5ex>[r]^-{\kappa}
\ar@{->}@<-0.5ex>[r]_{1 }& HH }$
is an equaliser diagram and   $\kappa \cdot \sigma=\sigma$ (see (\ref{c-diagrams})), there is a unique morphism
$j: HH \to \ol{G}$ such that $\ol{i} \cdot j=\sigma$. Then
$t_{\phi_{\ul{\bH}}(a)}=\ol{p}_a \cdot \sigma_a =\ol{p}_a \cdot \ol{i}_a \cdot j_a=j_a$  and the result follows.
\end{proof}

Symmetrically, we have:

\begin{proposition}\label{coaction1} In the situation described in \ref{comp-func}, the monad morphism
$t: \bT \to \phi^{\ol{\bH}}U^{\ol{\bH}}$ induced by the right triangle in (\ref{triangle1})
(see Section \ref{com.adj.func.}), has for its $(a,\theta)$-component
$$t_{(a,\theta)}:T(a,\theta) \xr{i'_{(a,\theta)}} H(a) \xr{H(\theta)} HH(a) \xr{m_a} H(a).$$
\end{proposition}

Our general results from  Section \ref{Prel} now yield:

\begin{proposition} \label{pair} Let $\bH=(\ul{\bH},\ol{\bH},\omega)$ be a
 compatible weak entwining on $\A$.
Then the functor $K:\A \to \Ab$ \emph{(}and hence also
 $K_{\omega}:\A \to \Aa$\emph{)} has a right adjoint if and only if,
for any $(a,h,\theta)\in \Aa$, the pair of morphisms
\begin{equation}\label{pair.d}
\xymatrix{a \ar@/^2pc/@{->}[rrr]^{\theta}
\ar[r]_{e_a} & H(a) \ar[r]_{\delta_a} & HH(a) \ar[r]_{H(h)}& H(a)}
\end{equation} has an equaliser in $\A$.
\end{proposition}
\begin{proof} Since the functor $U^\bG: \Ab \to \A$ is clearly (pre)comonadic, it follows from Theorem \ref{Dubuc2} that
the functor $K:\A \to \Ab$ admits a right adjoint if and only if for any $((a,h),\nu)\in \Ab$, the pair
$$
\xymatrix{a\ar@{->}@<0.5ex>[rr]^-{ \nu}
\ar@{->}@<-0.5ex>[rr]_-{\beta_{(a,\,h)}}&& G(a,h),}
$$
where $\beta:U_{\ul{\bH}} \to U_{\ul{\bH}} G$ is the right $\bG$-comodule structure on
$U_{\ul{\bH}}: \A_{\ul{\bH}} \to \A$ induced by the triangle (\ref{triangle1}) (see Section \ref{com.adj.func.}), has an
equaliser in $\A$, which -- since $i: G \to HU_{\ul{\bH}}$ is a (split) monomorphism -- is the case if and only if the  pair
$$
\xymatrix{a\ar@{->}@<0.5ex>[rr]^-{i_{(a,\,h)}\cdot \nu}
\ar@{->}@<-0.5ex>[rr]_-{i_{(a,\,h)} \cdot\beta_{(a,\,h)}}&& H(a)}
$$
has one. According to Propositions \ref{coaction} and
\ref{com.adj.func.}, $\beta_{(a,\,h)}$ is the composite
$$a \xr{\;e_a\;}H(a) \xr{\;\delta_a\;}HH(a) \xr{p_{_{(H(a),\,m_a)}}} G(H(a),m_a) \xr{G(h)}  G(a,h).$$
Since $\kappa\cdot \delta=\delta$ by (\ref{c-diagrams}) and  $\Gamma= i \cdot p$, it follows by naturality of $i$ that the diagram
$$
\xymatrix @C=.3in @R=.7in{a \ar@/^2pc/@{->}[rrrrr]^-{\beta_{(a,\,h)}}
 \ar[r]_-{e_a} & H(a)
\ar@/_2.4pc/@{->}[rrrd]_-{\delta_a}\ar[r]^-{\delta_a} & HH(a)
\ar[drr]_{\kappa_a=\Gamma_{(H(a),m_a)}} \ar[rr]^-{p_{(H(a), m_a)}} && G(H(a), m_a)
\ar[r]_-{G(h)} \ar[d]^{i_{(H(a),m_a)}}& G(a,h) \ar[d]^{i_{(a,\,h)}}\\
&&&& HH(a) \ar[r]_{H(h)}& H(a).}
$$
is commutative. So we have
$$i_{(a,\, h)} \cdot \beta_{(a,\,h)}=H(h) \cdot \delta_a \cdot e_a.$$
Thus, the functor $K:\A \to \Ab$ has a right adjoint if and only if for any
$((a,h),\nu)\in \Ab$, the pair of morphisms
$$
\xymatrix{a  \ar@{->}@<0.5ex>[rr]^-{i_{(a,h)} \cdot  \,\nu}
\ar@{->}@<-0.5ex>[rr]_{H(h) \cdot \,\delta_a \cdot \, e_a}&& H(a) }
$$ has an equaliser.
Recalling that $\Phi:\Aa \to \Ab$ is an isomorphism of categories and
$\Phi^{-1}((a,h),\nu)=(a,h,i_{(a,\,h)} \cdot  \nu)$ gives
the desired result.
\end{proof}

Symmetrically, we have:
\begin{proposition} \label{pair1}Let $\bH=(\ul{\bH},\ol{\bH},\omega)$
be a 
compatible weak entwining
on $\A$. Then the functor $K':\A \to \Ac$ \emph{(}and hence also
$K_{\omega}:\A \to \Aa$\emph{)} has a left adjoint if and only if
for any $(a,h,\theta)\in \Aa$, the pair of morphisms
\begin{equation}\label{pair.d1}
\xymatrix{H(a) \ar@/^2pc/@{->}[rrr]^-{h}
\ar[r]_-{H(\theta)} & HH(a) \ar[r]_-{m_a} & H(a) \ar[r]_-{\ve_a}& a}
\end{equation} has a coequaliser in $\A$.
\end{proposition}

Symmetric to \ref{w-entw} one may consider

\begin{thm}\label{w-bim}{\bf Weak comonad monad entwinings.} \em
 For a natural transformation $\oom :\ol H\ul H\to \ul H\ol H$,
  define the natural transformation
$$\begin{array}{rl}
 \oxi:& H\xr{He} HH \xr{\oom} HH \xr{H\ve} H.\\
\end{array}
$$

$(\ol H,\ul H, \oom)$ is called a {\em weak entwining}
  (from comonad  $\ol \bH$ to monad $\ul \bH$) provided
 \begin{equation}\label{ent-2}
\begin{array}[t]{lrl}\hspace{-7mm}
{\rm (i)} &
\oom\cdot Hm  = mH\cdot H\oom\cdot  \oom H,&
 H\delta \cdot \oom = \oom H \cdot H\oom \cdot \delta H, \\[+1mm]
\hspace{-7mm}
{\rm (ii)} &
  \oom\cdot He =\oxi H\cdot \delta  ,& H\ve  \cdot \oom = \oxi H\cdot \delta,
 \end{array}
\end{equation}
and is said to be compatible if
\begin{equation}\label{compatible-2}
\delta \cdot m=  mH\cdot H\oom \cdot  \delta H.
\end{equation}
Here we get
\begin{equation}\label{ent-equal}
 \oxi*\oxi=\oxi,\quad 1*\oxi = 1\, .
\end{equation}
\end{thm}

Certainly, the theory for this notion will be similar to that for monad comonad entwinings.
However, the mixed bimodules (as in \ref{mix-bim}) do not play the same role here
but are to be replaced  by liftings to Kleisli categories. Nevertheless, comonad monad
entwinings will enter the picture in the next section.

\section{Weak braided bimonads}\label{t-bim}

In the theory of Hopf algebras $H$ over a field $k$, the twist map
for $k$-vector spaces $M,N$,
$\tw_{M,N}:M\ot_k N\to N\ot_k M$,   plays a crucial
part. In particular it helps to commute $H\ot_k-$ with itself by
$\tw_{H,H}:H\ot_kH\to H\ot_kH$. Generalising this to monoidal categories,
often a {\em braiding} is required, that is, a condition on the whole category.
It was observed (e.g. in \cite{MW-Bim}) that it can be enough to have such a twist
only for the functor $H$ under consideration, that is,
a natural isomorphism $\tau:HH\to HH$ satisfying the Yang-Baxter equation.
For the study of {\em weak  braided Hopf algebras},
 Alonso \'Alvarez e.a. suggested in \cite[Definition 1.2]{AV-JA}  to
 consider, for any object $D$ in a monoidal category, a
{\em weak Yang-Baxter operator} $t_{D,D}:D\ot D\to D\ot D$, which is
not necessarily invertible but only regular. Here we take up this notion and formulate it for any functor on an arbitrary category.

\begin{thm}\label{weak-YB} {\bf Weak Yang Baxter operator.} \em
Given an endofunctor $H:\A\to \A$, a pair of natural transformations
 $\tau, \tau':HH\to HH$
is said to be a 
{\em weak YB-pair} provided
 the following equalities hold:
\begin{equation}
\tau\cdot \tau'\cdot \tau = \tau, \quad \tau'\cdot \tau\cdot \tau' = \tau', \quad
 \tau\cdot \tau'=\tau'\cdot \tau,
 \end{equation}
\begin{equation}\label{YB-equ}
\begin{array}[t]{c}
H\tau \cdot \tau H\cdot H\tau =  \tau H \cdot H\tau  \cdot  \tau H,\\[+1mm]
 H\tau' \cdot \tau' H\cdot H\tau' =  \tau' H \cdot H\tau'  \cdot  \tau H',
\end{array}
\end{equation}
and for $\nabla:=\tau\cdot \tau'$,
\begin{equation}\label{YB-2}
\begin{array}[t]{rl}
\tau H\cdot H\nabla = H\nabla \cdot\tau H, &
H\tau  \cdot  \nabla H = H\tau \cdot\nabla H, \\[+1mm]
\tau' H\cdot H\nabla = H\nabla \cdot\tau' H, &
H\tau'  \cdot  \nabla H = H\tau' \cdot\nabla H .
\end{array}
\end{equation}

The conditions in (\ref{YB-equ}) are the usual Yang-Baxter equations for $\tau$ and $\tau'$, respectively.
The equations in (\ref{YB-2}) 
are obviously satisfied if $\tau'=\tau^{-1}$ and in this case $\tau$ is known as {\em Yang-Baxter operator}.
\end{thm}

\begin{thm}\label{def-bialg} {\bf Definition.} \em
Let  $\ul{\bH}=(H, m, e)$ a monad,
$\ol{\bH}=(H, \delta,\varepsilon)$ a comonad on $\A$,
and $\tau,\tau':HH\to HH$ a weak YB-pair with $\nabla:=\tau\cdot\tau'$.
The triple
$\bH=(\ul{\bH},\ol{\bH},\tau)$ is called a \emph{weak braided bimonad} (or \emph{weak $\tau$-bimonad}) provided
\begin{zlist}
\item  $m\cdot \nabla=m$,\;  $ \nabla\cdot \delta =\delta$;
\item  $\nabla \cdot He = \tau\cdot eH$, \,  $H \ve \cdot \nabla = \ve H\cdot \tau$,\,
$\nabla \cdot eH = \tau\cdot He$, \,  $\ve H \cdot \nabla =H \ve \cdot \tau$;
\item  $\delta H\cdot \tau = H\tau  \cdot  \tau H\cdot H\delta  $,   \;
   $\tau\cdot mH =Hm\cdot \tau H\cdot H\tau$;
\item  $H\delta\cdot \tau = \tau H\cdot H\tau \cdot \delta H$,   \;
   $\tau\cdot Hm =mH\cdot H\tau \cdot  \tau H$;
  \item $\delta \cdot m =mm \cdot H\tau H \cdot \delta\delta$;
  \item 
 $\ve \ve \cdot m m \cdot H\delta H \, = \,
    \ve \cdot m \cdot mH \, = \, \ve \ve \cdot mm \cdot H \tau'H \cdot H \delta H$;
  \item 
    $HmH \cdot \delta \delta \cdot ee\, =\,
    \delta H \cdot \delta \cdot e\,=\, HmH \cdot H \tau'H \cdot \delta \delta \cdot ee$.
\end{zlist}
\end{thm}

For vector space categories and (finite dimensional) tensor functors $H\ot -$ with
$\tau$ the twist map, these  conditions were introduced in \cite[Definition 1]{BNS}.
For monoidal categories and monoidal functors the conditions are those for
  a {\em weak braided bialgebra} introduced and studied by
  Alonso \'Alvarez e.a. \cite{AV-JA,AVR} and we can - and will - freely use
essential parts of their results in our situation.
Note that if $\nabla$ is the identity of $H$, the conditions (1)--(4) in the definition
describe the invertible double entwinings considered in \cite{MW}.

The following observations provide the key to our previous results.

\begin{thm}\label{more-nat} {\bf Entwinings.} \em
Given the data from Definition \ref{def-bialg}, define
\begin{align*}
\omega&:HH \xr{\delta H} HHH \xr{H\tau} HHH \xr{mH} HH, \\
\ol{\omega}&:HH \xr{H  \delta} HHH \xr{\tau H} HHH \xr{H m} HH  .
\end{align*}
  From the Sections \ref{w-entw} and \ref{w-bim}  we get the natural transformations
$$\xi : H\xr{eH} HH \xr{\omega} HH \xr{\ve H}H,\quad
\ol{\xi}  : H\xr{He} HH \xr{\ol{\omega}} HH \xr{H\ve } H,$$
 and recalling  $\sigma$ and $\ol\sigma$ from \ref{fusion}, we define
$$\chi  : H\xr{eH} HH \xr{\sigma} HH \xr{H\ve}H,\quad
\ol{\chi}  : H\xr{He} HH\xr{\ol{\sigma}} HH \xr{\ve H}H. $$
\end{thm}

With these notions we collect the basic identities proved in \cite{AVR}.

\begin{proposition} 
\label{AVR} Let $\bH=(\ul{\bH},\ol{\bH},\tau)$ be a weak braided bimonad on  $\A$.  Then
\begin{zlist}
  \item  $\xi$, $\ol{\xi}$, $\chi$ and $\ol{\chi}$ are idempotent  (w.r.t. composition)
and respect  unit and counit of $\bH$;
 \item $\xi \cdot m \cdot  H\xi=\xi\cdot m$\,  and\: $\ol{\xi} \cdot m \cdot \ol{\xi} H = \ol{\xi}\cdot m$;
  \item $  H\xi \cdot \delta \cdot \xi  =\delta \cdot \xi $\,  and\: $  \ol{\xi}H  \cdot \delta \cdot \overline{\xi}  =\delta \cdot \overline{\xi}$;
  \item $\sigma \cdot eH=\chi H \cdot \delta$\, and\:  $\ol\sigma \cdot He=H \overline{\chi}  \cdot \delta$;
  \item $H\ve \cdot \sigma=m \cdot H\chi$\,   and\:
   $\ve H \cdot \ol\sigma=m \cdot \overline{\chi}H$;
  \item  $\xi\cdot\chi=\xi$,\; $\xi\cdot \ochi=\ochi$,\; $\chi\cdot \xi=\chi$,\; $\ochi\cdot \xi =\xi$,\\
  $\ol{\xi} \cdot \chi=\chi$, \;  $\oxi\cdot \ochi= \oxi$,\;    $\chi \cdot \ol{\xi}=\ol{\xi}$,\;
   $\ochi\cdot \oxi= \oxi$.
\end{zlist}
\end{proposition}
\begin{proof} (1) is shown in \cite[Proposition 2.9]{AVR};
(2), (3) are from \cite[Proposition 2.14]{AVR};
(4) is shown in  \cite[Proposition 2.6]{AVR}, (5)  in \cite[Proposition 2.4]{AVR},
and (6)  in \cite[Proposition 2.10]{AVR}.
\end{proof}

The following shows the way to apply results from the preceding section.

\begin{thm}\label{entw-str}{\bf Proposition.}
Let $\bH=(\ul{\bH},\ol{\bH},\tau)$ be a weak braided bimonad. 
Then
\begin{blist}
\item[$\bullet$]  $(\ul{\bH},\ol{\bH},\omega)$ is a compatible  weak
(monad comonad) entwining;
\item[$\bullet$]  $(\ol{\bH},\ul{\bH},\oom)$ is a compatible  weak (comonad monad)
entwining.
\end{blist}
\end{thm}
\begin{proof}  As easily seen,  condition (5) in \ref{def-bialg} yield
the equalities  (\ref{compatible}) and (\ref{compatible-2}) and
 also implies  (\ref{ent}(i)) and (\ref{ent-2}(i)) for  $\omega$  and $\oom$,
respectively (e.g. \cite{BBW}, \cite{MW}).
Now Propositions 2.3 and 2.5 in \cite{AVR} show the equations in
(\ref{ent})(ii) and  (\ref{ent-2})(ii).
\end{proof}

Direct inspection yields the technical observation:
\begin{lemma}\label{idempotents} Suppose that $f, g: X \to X$ are idempotents in an arbitrary category such that
$fg=g$ and $gf=f$. If $X \xr{p_f}X_f \xr{i_f}X$ (resp. $X \xr{p_g}X_g \xr{i_g}X$) is a splitting
of the idempotent $f$ (resp. $g$), then
$$
\xymatrix { X_g \ar[r]^-{i_g} & X \ar@{->}@<0.6ex>[r]^{f} \ar@ {->}@<-0.6ex> [r]_{1} &X} \qquad
(\text{resp}. \,\,\xymatrix { X_f \ar[r]^-{i_f} & X \ar@{->}@<0.6ex>[r]^{g} \ar@ {->}@<-0.6ex> [r]_{1} &X)}
$$ is a (split) equaliser diagram, while
$$
\xymatrix { X \ar@{->}@<0.7ex>[r]^{f} \ar@ {->}@<-0.7ex> [r]_{1} &X  \ar[r]^-{p_g} &X_g} \qquad
(\text{resp}. \,\,\xymatrix { X \ar@{->}@<0.7ex>[r]^{g} \ar@ {->}@<-0.7ex> [r]_{1} &X \ar[r]^-{p_f}&X_f )}
$$ is a (split) coequaliser diagram.
\end{lemma}

Henceforth we work over a Cauchy complete category $\A$ and suppose that
\begin{equation}\label{split-xi}
\xymatrix{ H \ar@{->>}[rd]_{q^{\ol{\xi}}}\ar[rr]^{\ol{\xi}} && H\\
 & H^{\ol{\xi}}\; \ar@{>->}[ru]_{\iota^{\ol{\xi}}} &}
\end{equation} is
a splitting of ${\ol{\xi}}$.

\begin{proposition}\label{equaliser} In the situation of Proposition \ref{pair}, the diagram
\begin{equation}\label{pair.dd}
\xymatrix{H^{\ol{\xi}} \ar[r]^{\iota^{\ol{\xi}}} & H \ar@/^2.1pc/@{->}[rrr]^{\delta}
\ar[r]_{eH} & HH \ar[r]_{\delta H} & HHH \ar[r]_{H m}& HH}
\end{equation} is a (split) equaliser in $[\A,\A]$.
\end{proposition}
\begin{proof} Since $H m \cdot \delta H \cdot eH=\chi H \cdot \delta$ by Proposition \ref{AVR} (4), we have to show that the diagram
$$\xymatrix{H^{\ol{\xi}} \ar[r]^-{\iota^{\ol{\xi}}}&H \ar@/^1.8pc/@{->}[rr]^{\delta}
\ar[r]_{\delta} & HH \ar[r]_{\chi H} & HH}
$$ is a split equaliser. Let us first show that the pair
\begin{equation}\label{pair.dd0}
\xymatrix{H \ar@/^2pc/@{->}[rr]^{\delta}
\ar[r]_{\delta} & HH \ar[r]_{\chi H} & HH}
\end{equation} is cosplit by the morphism $H \ve : HH \to H$. Indeed, since $H \ve \cdot \delta =1$ and
$H \ve \cdot \chi H \cdot \delta=\chi \cdot H \ve \cdot \delta=\chi$, it remains to show that
$\chi H \cdot \delta \cdot \chi=\delta \cdot \chi.$ For this, consider the diagram
$$\xymatrix @R=.5in{
H\ar@/_3pc/@{->}[dddd]_{\chi H \cdot\,\delta} \ar[rrr]^-{\chi} \ar[d]^{eH}&&&  H \ar[d]_{eH} \ar@/^3pc/@{->}[ddd]^{\chi H \cdot \,\delta}\\
HH \ar@{}[rrru]|{(1)}\ar[dd]^{\delta H}\ar[rd]^{\delta H}\ar[rrr]^-{H \chi}&&& HH \ar[d]_{\delta H}\\
& HHH \ar@{}[ru]|{(2)}\ar[d]^{H\delta H}\ar[rr]^-{HH \chi}&& HHH \ar[d]_{Hm}\\
HHH \ar@{}[rd]|{(5)}\ar@{}[ru]|{(3)}\ar[d]^{Hm}\ar[r]_{\delta HH}& HHHH \ar@{}[rru]|{(4)} \ar[r]^-{HHm}& HHH
\ar@{}[d]|{(6)}\ar[r]^-{HH \ve}& HH\\
HH \ar[rru]_-{\delta H} \ar[rr]_-{H \ve}&&H \ar[ru]_-{\delta}&}$$
in which the
\begin{itemize}
  \item regions (1), (2), (5) and (6) commute by naturality of composition;
  \item region (3) commutes by coassociativity of $\delta$;
  \item region (4) commutes by Proposition \ref{AVR} (5);
  \item the curved regions commute by \ref{AVR} (4).
\end{itemize} Hence the whole diagram commutes, implying
$$\chi H \cdot \delta \cdot \chi=\delta \cdot H\ve \cdot \chi H \cdot \delta=\delta \cdot \chi.$$
So the pair (\ref{pair.dd0})
is cosplit by the morphism $H \ve$ and hence one finds its equaliser by splitting the idempotent $\chi=H\ve \cdot \chi H \cdot \delta$.
But since $\ol{\xi} \cdot \chi=\chi$ and $\chi \cdot \ol{\xi}=\ol{\xi}$ by  Proposition \ref{AVR} (6) and  $\iota^{\ol{\xi}} \cdot q^{\ol{\xi}}$ is
the splitting of the idempotent $\ol{\xi}$ (see (\ref{split-xi})), it follows from Lemma \ref{idempotents} that (\ref{pair.dd}) is a (split)
equaliser diagram.
\end{proof}

Symmetrically, we have:
\begin{proposition}\label{coequaliser} In the situation of Proposition \ref{pair1}, the diagram
\begin{equation*}\label{pair.dd1}
\xymatrix{HH \ar@/^2pc/@{->}[rrr]^-{m}
\ar[r]_-{H\delta} & HHH \ar[r]_-{m H} & HH \ar[r]_-{\ve H}& H \ar[r]^-{q^{\ol{\xi}}}& H^{\ol{\xi}}}
\end{equation*} is a (split) coequaliser diagram in $[\A,\A]$.
\end{proposition}

Since $Hm \cdot \delta H \cdot eH=\chi H\cdot \delta$
and $\ve H \cdot mH \cdot H\delta =m \cdot \ol{\chi}H$ by Proposition \ref{AVR}(4) and (5),
Propositions \ref{equaliser} and \ref{coequaliser} immediately yield:

\begin{corollary} \label{corollary} Let $\bH=(\ul{\bH},\ol{\bH},\tau)$ be a weak
braided bimonad $\A$. Then
$$\xymatrix @C=0.6in @R=0.5in{H^{\ol{\xi}} \ar[r]^-{\iota^{\ol{\xi}}}&H \ar@/^1.5pc/@{-->}[l]^{q^{\ol{\xi}}}
\ar@/^2pc/@{->}[rr]^{\delta}
\ar[r]_{\delta} & HH \ar[r]_{\chi H} & HH \ar@/^2.2pc/@{-->}[ll]^{H \ve}}
$$ is a (split) equaliser yielding the monad
\begin{center}
$\ul{\bH}\,^{\ol{\xi}}=(H^{\ol{\xi}}, m^{\ol{\xi}}, e^{\ol{\xi}})$\; with\; $m^{\ol{\xi}}=q^{\ol{\xi}} \cdot m \cdot
(\iota^{\ol{\xi}}\iota^{\ol{\xi}})$\; and \;$ e^{\ol{\xi}}=q^{\ol{\xi}} \cdot e$,
\end{center}
 and
$$\xymatrix @C=0.6in @R=0.5in{HH \ar@/^2pc/@{->}[rr]^{m}
\ar[r]_{\ol{\chi}H} & HH \ar[r]_{m} & H \ar@/^2pc/@{-->}[ll]^{He}\ar[r]^-{q^{\ol{\xi}}}& H^{\ol{\xi}}\ar@/^1.5pc/@{-->}[l]^{\iota^{\ol{\xi}}}}
$$ is a (split) coequaliser yielding the comonad
\begin{center}
$\ol{\bH}\,_{\ol{\xi}}=(H^{\ol{\xi}}, \delta_{\ol{\xi}}, \ve_{\ol{\xi}})$ \; with \; $H_{\ol{\xi}}=H^{\ol{\xi}}$,
$\delta_{\ol{\xi}}=(q^{\ol{\xi}} q^{\ol{\xi}}) \cdot \delta \cdot \iota^{\ol{\xi}}$  \;and \; $\ve_{\ol{\xi}}=\ve \cdot \iota^{\ol{\xi}}$.
\end{center}
\end{corollary}

The next result provides the technical data to show  Frobenius separability.

\begin{thm}{\bf Lemma.}\label{lemma}
 Let $\bH=(\ul{\bH},\ol{\bH},\tau)$ be a  weak braided bimonad and
consider  the composite
$$\upsilon:1 \xr{e}H \xr{\delta}HH \xr{q^{\ol{\xi}}q^{\ol{\xi}}}H^{\ol{\xi}}H^{\ol{\xi}}.$$
For this, $m^\oxi \cdot \upsilon = e^\oxi$ and
one has commutativity of the   diagrams
$$
\xymatrix @C=0.4in @R=0.4in{
H^{\ol{\xi}} \ar@{..>}[rrd]^{\delta_{\ol{\xi}}}\ar[rr]^-{\upsilon H^{\ol{\xi}}} \ar[d]_{\iota^{\ol{\xi}}}&& H^{\ol{\xi}} H^{\ol{\xi}} H^{\ol{\xi}} \ar[d]^{H^{\ol{\xi}}m^{\ol{\xi}}}\\
H \ar[r]_\delta & HH \ar[r]_-{q^{\ol{\xi}}q^{\ol{\xi}}}& H^{\ol{\xi}}H^{\ol{\xi}},}
\quad
\xymatrix @C=0.4in @R=0.4in{
H^{\ol{\xi}} \ar@{..>}[rrd]^{\delta_{\ol{\xi}}}\ar[rr]^-{ H^{\ol{\xi}}\upsilon} \ar[d]_{\iota^{\ol{\xi}}}&& H^{\ol{\xi}} H^{\ol{\xi}} H^{\ol{\xi}} \ar[d]^{m^{\ol{\xi}}H^{\ol{\xi}}}\\
H \ar[r]_\delta & HH \ar[r]_-{q^{\ol{\xi}}q^{\ol{\xi}}}& H^{\ol{\xi}}H^{\ol{\xi}}.}
$$
 \end{thm}
\begin{proof}
The diagram
$$
\xymatrix @C=0.4in @R=0.4in{
1 \ar[ddr]_e\ar[r]^e&H \ar[d]_{He}\ar[r]^\delta &HH \ar[r]^{HHe}& HHH \ar[d]^{H\oom}\\
& HH \ar[d]^m\ar[rru]_{\delta H}&HH   \ar[d]^{H\varepsilon}  &
  HHH\ar[l]^{mH} \ar[d]^{HH\varepsilon}\\
&H\ar[ru]^\delta \ar@{=}[r]&H &HH \ar[l]^m \\
}$$
commutes by naturality and since $\oom$ is an entwining
 (Proposition \ref{entw-str}).
Now the equation $\oxi\cdot m\cdot \oxi H =\oxi\cdot m$ (Proposition \ref{AVR}(2))
yields the commutative  diagram
$$ 
\xymatrix @C=0.4in @R=0.4in{
1 \ar[rrd]_{e}\ar[r]^e &H \ar[r]^{\delta} & HH\ar[r]^{H\oxi}   &
   HH\ar[ld]^m\ar[r]^{ \oxi H} & HH \ar[d]^m \\
&& H  \ar[r]_\oxi & H &\ar[l]^\oxi   H
}$$
and the splitting $\oxi=\iota^\oxi\cdot q^\oxi$ implies
 $m^\oxi \cdot \upsilon = e^\oxi$.
 \smallskip

In the diagram
$$
\xymatrix  @C=0.4in @R=0.3in{
H^{\ol{\xi}} \ar[d]_{\iota^{\ol{\xi}}}\ar[r]^-{eH^{\ol{\xi}}} & HH^{\ol{\xi}} \ar[d]_-{H \iota^{\ol{\xi}}}\ar[r]^-{\delta H^{\ol{\xi}}}
& HHH^{\ol{\xi}} \ar[d]_-{H {\ol{\xi}}H^{\ol{\xi}}} \ar[r]^-{q^{\ol{\xi}}HH^{\ol{\xi}}}
& H^{\ol{\xi}} H H^{\ol{\xi}} \ar[rd]_{H^{\ol{\xi}}\,{\ol{\xi}}H^{\ol{\xi}}}\ar[r]^-{H^{\ol{\xi}} q^{\ol{\xi}} H^{\ol{\xi}}}&H^{\ol{\xi}}H^{\ol{\xi}}H^{\ol{\xi}} \ar[d]^{H^{\ol{\xi}}\,\iota^{\ol{\xi}}H^{\ol{\xi}}}\\
H \ar[ddd]_{\delta} \ar[r]_-{eH}&HH \ar[d]^{\delta H}& HHH^{\ol{\xi}} \ar[d]^{HH \iota^{\ol{\xi}}}\ar[rr]_-{q^{\ol{\xi}}HH^{\ol{\xi}}} && H^{\ol{\xi}} H H^{\ol{\xi}} \ar[d]^{H^{\ol{\xi}} H\iota^{\ol{\xi}}}\\
& HHH \ar@{}[rdd]|{(b)}\ar[dd]^{Hm}\ar[r]_-{H{\ol{\xi}}H}&HHH \ar[d]^{Hm}\ar[rr]_-{q^{\ol{\xi}} HH}&&H^{\ol{\xi}} HH \ar[d]^{H^{\ol{\xi}}\, m}\\
&& HH \ar[d]^{H q^{\ol{\xi}}}\ar[rr]_{q^{\ol{\xi}}H}&& H^{\ol{\xi}}H \ar[d]^{H^{\ol{\xi}}q^{\ol{\xi}}}\\
HH \ar@{}[uur]|{(a)}\ar[r]_{\chi H}& HH  \ar[r]_{Hq^{\ol{\xi}}}& HH^{\ol{\xi}} \ar[rr]_{ q^{\ol{\xi}} H^{\ol{\xi}}}&&H^{\ol{\xi}}H^{\ol{\xi}}
}$$
\begin{itemize}
  \item the triangle commutes by the splitting of $\ol{\xi}$;
  \item region (a) commutes by Proposition 6.4(4);
  \item region (b) commutes because
 Proposition 6.4(2) induces
  $q^{\ol{\xi}} \cdot m \cdot \ol{\xi}H=q^{\ol{\xi}} \cdot m$;
  \item the other regions commute by naturality.
\end{itemize}
Since $\chi H \cdot \delta \cdot \iota^{\ol{\xi}}=\delta \cdot \iota^{\ol{\xi}}$ by Proposition \ref{equaliser}, we obtain from this commutativity of the
left hand diagram.
\smallskip

Next, in the diagram
$$
\xymatrix @C=0.4in @R=0.3in{
H^{\ol{\xi}} \ar[d]_{\iota^{\ol{\xi}}}\ar[r]^-{H^{\ol{\xi}}e} &
H^{\ol{\xi}}H \ar[d]_-{ \iota^{\ol{\xi}}H}\ar[r]^-{ H^{\ol{\xi}}\delta}
& H^{\ol{\xi}}HH \ar[d]_-{\iota^{\ol{\xi}}HH} \ar[r]^-{H^{\ol{\xi}}q^{\ol{\xi}}H}\ar[rd]_{H^{\ol{\xi}}{\ol{\xi}}H}
& H^{\ol{\xi}} H^{\ol{\xi}}H \ar[r]^-{H^{\ol{\xi}} H^{\ol{\xi}} q^{\ol{\xi}} }\ar[d]^{H^{\ol{\xi}}\,\iota^{\ol{\xi}}H}&H^{\ol{\xi}}H^{\ol{\xi}}H^{\ol{\xi}} \ar[d]^{H^{\ol{\xi}}\,\iota^{\ol{\xi}}H^{\ol{\xi}}}\\
H \ar@{}[rrrd]|{(a)}\ar[dd]_{\delta}\ar[rd]_{He}\ar[r]^-{He}&HH \ar[r]_{H\delta}& HHH \ar[rd]_{H{\ol{\xi}}H} &H^{\ol{\xi}}HH
\ar[d]^{\iota^{\ol{\xi}} HH}& H^{\ol{\xi}} H H^{\ol{\xi}} \ar[d]^{\iota^{\ol{\xi}} HH^{\ol{\xi}}}\\
& HH \ar[rr]_{H \delta} &
& HHH \ar[d]^{mH} & HHH^\oxi \ar[d]^{m H^\oxi}  \\
HH\ar[rrr]^{H\ol{\chi}}\ar[d]_{q^\oxi H} \ar@{}[rru]|{(b)} &&& HH
\ar[d]^{ q^{\ol{\xi}} H}\ar[r]_{Hq^{\ol{\xi}}}& HH^{\ol{\xi}} \ar[d]^{q^{\ol{\xi}}H^{\ol{\xi}}}\\
H^\oxi H\ar[rrr]_{H \ol{\chi}}&&&
H^\oxi H  \ar[r]_{H^{\ol{\xi}} q^{\ol{\xi}}}&H^{\ol{\xi}}H^{\ol{\xi}}
}$$
\begin{itemize}
  \item the top triangle commutes by the splitting of $\ol{\xi}$;
  \item region (a) commutes since
from Proposition \ref{AVR}(1) and (3) we get
$${\ol{\xi}}H \cdot \delta \cdot e  = {\ol{\xi}}H \cdot \delta \cdot {\ol{\xi}}\cdot e=
\delta \cdot {\ol{\xi}}\cdot e= \delta \cdot e;$$
  \item region (b) commutes by Proposition \ref{AVR}(4);
   \item the other regions commute by naturality.
\end{itemize} Thus the whole diagram is commutative, and in the light of the equality $q^{\ol{\xi}}\cdot \ol{\chi}= q^{\ol{\xi}}$,
derived from $\ol{\xi}\cdot \ol{\chi}= \ol{\xi}$ in Proposition \ref{AVR}(6),
commutativity of the right hand diagram follows.
\end{proof}

The following generalises \cite[Proposition 4.2]{Sch} and
\cite[Proposition 1.6]{Szl-FQ}  (see also \cite[Proposition 4.4]{BCJ}).

\begin{proposition}\label{sep.frobenius2} Let $\bH=(\ul{\bH},\ol{\bH},\tau)$ be a
 weak braided bimonad on $\A$.
Then  the  quintuple $(H^\oxi, m^\oxi,e^\oxi;\delta_\oxi,\ve_\oxi)$ is a Frobenius separable (co)monad.
\end{proposition}

\begin{proof} Using the results of Lemma \ref{lemma} it is easy to verify that the diagram
corresponding to (\ref{sep-mon-d}) is commutative and that $ m^{\ol{\xi}} \cdot \delta_{\ol{\xi}}=1_{H^\oxi}$.

The statement about the comonad $\ol{\bH}\,_{\ol{\xi}}$ holds by symmetry.
\end{proof}

\begin{proposition}\label{gen.monad} Let $\bH=(\ul{\bH},\ol{\bH},\tau)$ be a
weak braided bimonad on $\A$
and suppose the functor $K_\omega:\A \to \Aa$ admits a right adjoint
$\ol{K}:\Aa \to \A$. Then $\ul{\bH}\,^{\ol{\xi}}=(H^{\ol{\xi}}, m^{\ol{\xi}}, e^{\ol{\xi}})$ is
the monad generated by the adjunction $K_\omega \dashv \ol{K}$
 and
$\iota^{\ol{\xi}}$ is the corresponding morphism of monads $\ul{\bH}\,^{\ol{\xi}} \to \ul{\bH}$ (as in Proposition \ref{monad.mor.}).
\end{proposition}
\begin{proof} Suppose there is an adjunction
$$\ol{\eta}\,,\, \ol{\ve} :K_\omega \dashv \ol{K}:\Aa \to \A ;$$
 consider the monad $(\ol{K} K_\omega, \ol{K}\ol{\ve}K_\omega, \ol{\eta})$ it generates on $\A$.
 Write $\iota : \ol{K} K_\omega \to \ul{\bH}$ for the corresponding morphism of monads
(as in Proposition \ref{monad.mor.}).
Since, by Proposition \ref{pair}, the functor $K_\omega:\A \to \Aa$ admits a right adjoint
if and only if for every $(a,h,\theta)\in \Aa$, the pair of morphisms (\ref{pair.d})
 has an equaliser in $\A$, and
since $K_{\omega}(a)=(H(a), m_a, \delta_a)$ for all $a \in \A$, it follows from
Proposition \ref{equaliser} that  $\ol{K}K_\omega=H^{\ol{\xi}}$ and that $\iota^{\ol{\xi}}=\iota$.

According to Theorem \ref{Dubuc2}, the unit of the adjunction $K_\omega \dashv \ol{K}$ is the (unique)
natural transformation $\ol{\eta}:1 \to \ol{K}K_\omega=H^\oxi$ such that the diagram
$$\xymatrix  @R=.3in @C=.4in{H^{\ol{\xi}}(a) \ar[r]^{\iota^{\ol{\xi}}_a}& H(a)\\&
 a \ar[lu]^-{\ol{\eta}_a} \ar[u]_{e_a}}
$$
commutes, for all $a \in \A$. It then follows that $\ol{\eta}_a=
q^{\ol{\xi}}_a \cdot \iota^{\ol{\xi}}_a \cdot \ol{\eta}_a
=q^{\ol{\xi}}_a \cdot e_a.$

\smallskip

Next, since for any $a \in \A$, $\ol{\ve}_{K_\omega(a)}=m_a \cdot H(\iota^{\ol{\xi}}_a)$ by
(\ref{counit}), it follows that
$$\ol{K}(\ol{\ve}_{K_\omega(a)}):\ol{K} K_\omega\ol{K}K_\omega(a)
=H^{\ol{\xi}} H^{\ol{\xi}}(a) \to \ol{K}K_\omega(a)=H^{\ol{\xi}}(a)$$
is the unique morphism making the diagram
$$\xymatrix @R=.3in @C=.5in{
H^{\ol{\xi}} H^{\ol{\xi}}(a) \ar[rr]^{(\iota^{\ol{\xi}})_{H^{\ol{\xi}}(a)}}
\ar[dd]_{\ol{K}(\ol{\ve}_{K_\omega(a)})}&& HH^{\ol{\xi}}(a) \ar[d]^{H(\iota^{\ol{\xi}}_a)}\\
&& HH(a) \ar[d]^{m_a}\\
H^{\ol{\xi}}(a) \ar[rr]_{\iota^{\ol{\xi}}_a} && H(a) }
$$ commute and we calculate (recall (\ref{split-xi}))
$$\ol{K}(\ol{\ve}_{K_\omega(a)})=
q^{\ol{\xi}}_a \cdot \iota^{\ol{\xi}}_a \cdot \ol{K}(\ol{\ve}_{K_\omega(a)})
=q^{\ol{\xi}}_a \cdot m_a \cdot
H(\iota^{\ol{\xi}}_a)\cdot \iota^{\ol{\xi}}_{H^{\ol{\xi}}(a)}
= q^{\ol{\xi}}_a \cdot m_a \cdot (\iota^{\ol{\xi}} \iota^{\ol{\xi}})_a.$$
This completes the proof.
\end{proof}

By symmetry, we also have:
\begin{proposition}\label{gen.comonad} Let $\bH=(\ul{\bH},\ol{\bH},\tau)$ be a
weak braided bimonad on $\A$
and suppose the functor $K_\omega:\A \to \Aa$ admits a left adjoint
$\ul{K}:\Aa \to \A$. Then
$\ol{\bH}\,_{\ol{\xi}}=(H^{\ol{\xi}}, \delta_{\ol{\xi}}, \ve_{\ol{\xi}})$
 is the comonad generated by the adjunction $ \ul{K}\dashv K_\omega$ and
 $q^\oxi$ is a comonad morphism $\ol{\bH} \to \ol{\bH}\,_{\ol{\xi}}$.
\end{proposition}

 According to Section \ref{canonical module}, the monad morphism
$\iota^{\ol{\xi}}:\ul{\bH}^{\ol{\xi}} \to \ul{\bH}$
equips $H$ with an $\ul{\bH}^{\ol{\xi}}$-bimodule structure, where the
left action is the composite
$$\rho_l: H^{\ol{\xi}}H \xr{\iota^{\ol{\xi}} H} HH \xr{m} H,$$ and the right action is the composite
$$\rho_r: HH^{\ol{\xi}} \xr{H\iota^{\ol{\xi}}} HH \xr{m} H.$$

\begin{proposition} \label{gamma-coeq}
For a weak braided bimonad $\bH=(\ul{\bH},\ol{\bH},\tau)$ on $\A$,
$\sigma=Hm \cdot \delta H$ coequalises the pair
$(\rho_r H, H \rho_l)$, i.e.,
 $$\sigma \cdot \rho_r H  = \sigma \cdot H \rho_l.$$
\end{proposition}
\begin{proof} Since (\ref{pair.dd}) is an equaliser diagram and since in Figure 1
\begin{itemize}
  \item  Diagram (1) commutes since $e$ is the unit of the monad $\bT$;
  \item  Diagrams (2) and (10) commute by Proposition \ref{entw-str};
  \item  Diagrams (3), (7) and (9)  commute by associativity of $m$;
  \item  Diagrams (4), (5), (6) and (8) commute by functoriality of composition,
\end{itemize}   one sees that
$\sigma \cdot \rho_r H  =Hm \cdot \delta H \cdot Hm \cdot H\iota^{\ol{\xi}} H=
\sigma \cdot  H \rho_l $.
\end{proof}

\begin{figure}[t] \label{figure}
\xymatrix@C=.6in @R=.6in{HH^{\ol{\xi}}H \ar[ddd]_{H\iota^{\ol{\xi}} H}
\ar[rr]^-{H\iota^{\ol{\xi}} H} &&
 H^3 \ar@{}[rrd]|{(2)} \ar[dl]_{HeH^2} \ar[d]^{H \delta H}\ar[r]^{mH}& H^2 \ar[r]^{\delta H} &
H^3 \ar@{}[rdd]|(.15){(3)} \ar[r]^{Hm}& H^2\\
& H^4 \ar@{}[l]|{(1)} \ar[ddl]_{mH^2} \ar[dd]_{H^2m}\ar[rd]_{H \delta H^2}& H^4 \ar[rd]|{H^2 m}
 \ar[r]_{\omega H^2} & H^4 \ar@{}[d]|{(4)} \ar[ru]_{HmH} \ar[r]_{H^2m} & H^3 \ar[ru]_{Hm} &\\
& & H^5 \ar@{}[r]|{(7)} \ar[rd]_{H^3m}\ar[u]|{H^2mH}& H^3 \ar@{}[rd]|{(8)} \ar[ru]_{\omega H}&
\ar@{}[r]|{(9)}&\\
H^3 \ar@{}[r]|{(5)} \ar[rd]_{Hm}& H^3 \ar@{}[drr]|{(10)}\ar@{}[ru]|{(6)}\ar[d]_{mH}
\ar[rr]_{H \delta H} && H^4 \ar[u]_{H^2m}\ar[r]_{\omega H^2} &
 H^4 \ar[uu]^{H^2m}\ar[r]^{HmH}& H^3 \ar[uuu]_{Hm}\\
& H^2 \ar[rrrru]_{\delta H}&&&&
} \caption{}\end{figure}

Suppose now that the tensor product $H \!\otimes_{\ul{\bH}^{\ol{\xi}}} \! H$ exists,
i.e., there is a coequaliser diagram
\begin{equation}\label{def_coeq}
\xymatrix @C=.5in{  HH^{\ol{\xi}}H \ar@{->}@<0.5ex>[r]^-{\rho_r H}
\ar@{->}@<-0.5ex>[r]_{H \rho_l}&   HH \ar[r]^-{l}&   H \!\otimes_{\ul{\bH}^{\ol{\xi}}} \! H\,,}
\end{equation} where $l=\text{can}_{\ul{\bH}^{\ol{\xi}}}^{H, H}$.
Note that, by Proposition \ref{action},
$H \!\otimes_{\ul{\bH}^{\ol{\xi}}} \! H$ has a right $\ul{\bH}$-module structure
such that $l$ is a morphism of right $\ul{\bH}$-modules. Moreover, since $\sigma$
coequalises  $\rho_r H$ and $H \rho_l$,
the composite $\ol{p}\,\cdot \,\sigma: HH \to \ol{G} $ (see (\ref{olG}))
also coequalises them, and since Diagram (\ref{def_coeq}) is a coequaliser,
there exists a unique natural transformation
$\gamma:H \!\otimes_{\ul{\bH}^{\ol{\xi}}} \! H \to \ol{G}$ making the diagram
\begin{equation}\label{gamma}
\xymatrix @R=.3in @C=.5in{  HH^{\ol{\xi}}H \ar@{->}@<0.5ex>[r]^-{\rho_r H}
\ar@{->}@<-0.5ex>[r]_{H \rho_l}& HH \ar[rd]_{\ol{p}\,\cdot \sigma}\ar[r]^-{l}&
 H \!\otimes_{\ul{\bH}^{\ol{\xi}}} \! H \ar@{-->}[d]^{\exists !\gamma}\\
&& \ol{G}}
\end{equation}
commute. It follows -- since $U_{\ul{\bH}}t \phi_{\ul{\bH}}=\ol{p}\cdot \sigma$
by Proposition \ref{coaction} 
 -- that the diagram
\begin{equation}\label{fund}
\xymatrix @R=.4in @C=.5in{HH \ar[rd]_{U_{\ul{\bH}}t \phi_{\ul{\bH}}}\ar[r]^{\sigma} \ar[d]_{l} &
 HH  \ar[d]^{\ol{p}}\\
H \!\otimes_{\ul{\bH}^{\ol{\xi}}} \! H \ar[r]_{\gamma} & \ol{G}}
\end{equation} commutes.
Precomposing this square with $He$ and using   $\sigma \cdot He=\delta$
(e.g. \cite[5.2]{MW-Bim}), we get
\begin{equation}\label{fund0}
\ol{p} \cdot \delta=\gamma \cdot l \cdot He.
\end{equation}

\begin{thm}\label{gamma'-coeq}{\bf $H$ as $H^\oxi$-bicomodule.} \em
The comonad morphism $q^{\ol{\xi}}:\ol{\bH} \to \ol{\bH}_{\ol{\xi}}$
equips $H$ with an $\ol{\bH}_{\ol{\xi}}$-bicomodule structure, where the
left coaction is the composite
$$\theta_l: H \xr{\delta} HH  \xr{q^{\ol{\xi}}H} H_{\ol{\xi}}H,$$
and the right coaction is the composite
$$\theta_l: H \xr{\delta} HH  \xr{Hq^{\ol{\xi}}} HH_{\ol{\xi}}.$$

Moreover, there is a unique natural transformation
$\gamma':\ol{T}\to H \!\otimes^{\ol{H}_{\ol{\xi}}}\!H$ making
the triangle in the diagram
\begin{equation}\label{gamma'}
\xymatrix @R=.3in @C=.5in{ H \!\otimes^{\ol{H}_{\ol{\xi}}}\!H \ar[r]^{\text{can}}&
HH \ar@{->}@<0.5ex>[r]^-{\theta_r H}
\ar@{->}@<-0.5ex>[r]_-{H \theta_l}& HH_{\ol{\xi}}H \\
\ol{T} \ar@{-->}[u]^{\exists !\gamma'} \ar[ru]_{\sigma\,\cdot \ol{i}\,'}&& }
\end{equation}
commute. Here $\ol{i}\,'$ and $\ol{T}$ are as defined in the diagram (\ref{olT}).
\end{thm}

\begin{proposition}\label{right}
Let $\bH=(\ul{\bH},\ol{\bH},\tau)$ be a weak braided bimonad on $\A$.
Viewing $H$ as a right $\ul{\bH}\,^{\ol{\xi}}$-module by  the structure map
$HH^{\ol{\xi}} \xr{H\iota^{\ol{\xi}}} HH \xr{m} H$, then
$q^{\ol{\xi}}:H \to H^{\ol{\xi}}$ is a morphism of right
$\ul{\bH}\,^{\ol{\xi}}$-modules.
\end{proposition}
\begin{proof}
For this, we have to show commutativity of the diagram
$$\xymatrix
{HH^{\ol{\xi}} \ar[rr]^-{q^{\ol{\xi}}H}\ar[d]_{H \iota^{\ol{\xi}}}&&
 H^{\ol{\xi}}H^{\ol{\xi}} \ar[d]^{m^{\ol{\xi}}}\\
HH \ar[r]_-{m}& H \ar[r]_-{q^{\ol{\xi}}}& H^{\ol{\xi}} .}
$$
Since $m^{\ol{\xi}}=q^{\ol{\xi}}\cdot m \cdot (\iota^{\ol{\xi}} \iota^{\ol{\xi}})$,
this diagram can be rewritten as
$$\xymatrix
{HH^{\ol{\xi}} \ar[r]^-{\ol{\xi}H}\ar[d]_{H \iota^{\ol{\xi}}}&
HH^{\ol{\xi}}\ar[r]^-{H \iota^{\ol{\xi}}} & HH \ar[d]^{m}\\
HH \ar[d]_-{m} \ar@{-->}[rru]_{\ol{\xi}H}&& H \ar[d]^{q^{\ol{\xi}}}\\
H \ar[rr]_{q^{\ol{\xi}}}&& H^{\ol{\xi}}}
$$
 and  in this diagram the triangle commutes by naturality of composition, and
  the trapezoid commutes, since $\iota^{\ol{\xi}}$ is a (split) monomorphism and
  ${\ol{\xi}} \cdot m \cdot {\ol{\xi}} H={\ol{\xi}} \cdot m$ by Proposition \ref{AVR}(2).
This completes the proof.
\end{proof}

Consider now the diagram
\begin{equation}\label{qtilde}
\xymatrix @R=.4in @C=.5in{
H H^{\ol{\xi}} H  \ar[d]_{q^{\ol{\xi}} H^{\ol{\xi}} H}\ar@{->}@<0.5ex>[r]^{\rho_r H}
\ar@ {->}@<-0.5ex>[r]_{H \rho_l}& HH \ar[d]^{q^{\ol{\xi}} H}
\ar[r]^-{l}& H\!\otimes_{\ul{\bH}^{\ol{\xi}}} \!H \ar@{..>}[dr]^-{\widetilde{q}}&\\
H^{\ol{\xi}} H^{\ol{\xi}} H  \ar@{->}@<0.5ex>[r]^{m^{\ol{\xi}} H}
\ar@ {->}@<-0.5ex> [r]_{H^{\ol{\xi}} \rho_l}& H^{\ol{\xi}}
H \ar@/_1.8pc/@{->}[rr]_{\rho_l}
\ar[r]^-{\iota^{\ol{\xi}} H}& HH \ar[r]^m& H}
\end{equation}
in which $q^{\ol{\xi}}H \cdot \rho_rH=m^{\ol{\xi}}H \cdot q^{\ol{\xi}}H^{\ol{\xi}}H$,
 since $q^{\ol{\xi}}$ is a morphism of
right $\ul{\bH}^{\ol{\xi}}$-modules by Proposition \ref{right},
$q^{\ol{\xi}}H \cdot H\rho_l=H^{\ol{\xi}}\rho_l \cdot q^{\ol{\xi}} H^{\ol{\xi}}H$,
because of naturality of composition, and the bottom row is a split coequaliser,
since the pair $(H, \rho_l)$ is a left $\ul{\bH}^{\ol{\xi}}$-
module (see Remark \ref{rightiso.3} (1)). It then follows

\begin{proposition}\label{righ.1}
Let $\bH=(\ul{\bH},\ol{\bH},\tau)$ be a weak braided bimonad on $\A$.
In the situation described above,
 there is a unique morphism $\widetilde{q}:H\!\otimes_{\ul{\bH}^{\ol{\xi}}} \!H \to H$
making the trapezoid in the diagram (\ref{qtilde}) commutative
and this is a morphism of right $\ul{\bH}$-modules.
\end{proposition}
\begin{proof} According to Section \ref{canonical module}, the morphisms
$q^{\ol{\xi}}H, \iota^{\ol{\xi}} H$ and $m$ are
morphisms of right $\ul{\bH}$-modules. Then  the composite
$m \cdot \iota^{\ol{\xi}} \cdot q^{\ol{\xi}}=\rho_l \cdot q^{\ol{\xi}}$
(and hence also $\widetilde{q} \cdot l$) is a morphism of right $\ul{\bH}$-modules,
implying -- since $l$ and $lH$ are both epimorphisms of right $\ul{\bH}$-modules --
that $\widetilde{q}$ is also a morphism of right $\ul{\bH}$-modules.
\end{proof}

\begin{proposition}\label{convolution} Suppose
$\gamma:H \!\otimes_{\ul{\bH}^{\,\ol{\xi}}} \! H \to \ol{G}$ in (\ref{gamma}) to be
an epimorphism. If the morphisms $f,g: H \to H$ are such that $f*1=g*1$, then $f*\xi=g*\xi$.
\end{proposition}
\begin{proof} If $f,g: H \to H$ are morphisms such that $f*1=g*1$, then
$$m \cdot fH \cdot \delta=m \cdot gH \cdot \delta$$
and since $\sigma \cdot He=\delta$ (e.g. \cite[(5.2)]{MW-Bim}), we have
$$m \cdot fH \cdot \sigma \cdot He=m \cdot gH \cdot \sigma \cdot He.$$
According to Section \ref{canonical module},
$fH$ and $gH$ can be seen as morphisms of the right $\ul{\bH}$-module
$(HH, Hm)$ to itself, while $m$ is a morphism from the right $\ul{\bH}$-module $(HH, Hm)$ to
the $\ul{\bH}$-module $(H, m)$. Moreover,  $\sigma$ is also a morphism of right
$\ul{\bH}$-modules (e.g. \cite[Section 5.1]{MW-Bim}). Thus the composites
$m \cdot fH \cdot \sigma$
and $m \cdot gH \cdot \sigma$ both are morphisms of right $\ul{\bH}$-modules. It
then follows from the right version of \cite[Lemma 3.2]{MW-Bim} that
$$m \cdot fH \cdot \sigma=m \cdot gH \cdot \sigma.$$
Next, since
$\sigma=\kappa \cdot \sigma=\ol{i} \cdot \ol{p} \cdot \sigma$ and
$\ol{p} \cdot \sigma=\gamma \cdot l$ (by (\ref{fund})), we have
$$m \cdot fH \cdot \ol{i} \cdot \gamma \cdot l=m \cdot gH \cdot \ol{i} \cdot \gamma \cdot l,$$
and   $l$ and $\gamma$ being epimorphisms  we get
$m \cdot fH \cdot \ol{i}=m \cdot gH \cdot \ol{i}$,
 thus
$$m \cdot fH \cdot \kappa=m \cdot gH \cdot \kappa.$$
Recalling that $\kappa \cdot He=\omega \cdot eH$ and
$\omega \cdot eH=H\xi \cdot \delta$
(see (\ref{c-diagrams}), (\ref{ent})(ii)), we get
$$
m \cdot fH \cdot \kappa \cdot He   = m \cdot fH \cdot \omega \cdot eH
  = m \cdot fH \cdot H\xi \cdot \delta
  = f * \xi,
$$
and similarly, one derives $m \cdot gH \cdot \kappa \cdot He=g * \xi.$
Thus, $f * \xi=g * \xi$.
\end{proof}

Since $\kappa$ is (clearly) a morphism of right $\ul{\bH}$-modules,
  Proposition  \ref{idem} yields

\begin{proposition} \label{morphism} The composite
$\ol{G}H \xr{\overline{i} H}HHH \xr{Hm}HH \xr{\overline{p}}\ol{G}$
makes $\ol{G}$ into a right $\ul{\bH}$-module such that
$\ol{p}: HH \to \ol{G}$ and $\ol{i}: \ol{G} \to HH$ both are morphisms
of right $\ul{\bH}$-modules.
\end{proposition}

\begin{corollary} \label{corollary.1}$\gamma:H \!\otimes_{\ul{\bH}^{\,\ol{\xi}}} \! H \to \ol{G}$
is a morphism of right $\ul{\bH}$-modules.
\end{corollary}
\begin{proof} Since, in  Diagram (\ref{fund}),  the morphisms $\sigma$, $\ol{p}$ and $l$
are all morphisms of right $\ul{\bH}$-modules (see \cite[5.1]{MW-Bim} and Proposition \ref{morphism}),
and since $l$ and $lH$ are both epimorphisms, it follows that $\gamma$ is also a morphism of
right $\ul{\bH}$-modules.
\end{proof}

\section{Weak braided Hopf monads}\label{t-Hopf}

In this section, we define a weak antipode for weak braided bimonads
$\bH=(\ul H,\ol H,\tau)$  and formulate various forms of a Fundamental Theorem.
The definition corresponds to that in \cite{BNS}, \cite{AVR} and in other papers on
(generalisations of) weak Hopf algebras.
 For the notations we refer to the preceding section.

\begin{thm}\label{def-S}{\bf Definition.}  \em
Given a weak  braided bimonad $\bH$,
 a natural transformation $S:H \to H$ is
called an \emph{weak antipode}, provided
$$ 1*S=\xi,\quad S * 1=\ol{\xi}, \quad S*1*S=S.$$
Since  $ \xi* 1=1=1*\oxi $ (see  (\ref{c-diagrams}), (\ref{ent-equal})), we also get $1*S*1=1$.

A weak braided bimonad $\bH$ with a weak antipode is called   a {\em weak braided Hopf monad} or
a {\em weak $\tau$-Hopf monad}.
\end{thm}

\begin{proposition}\label{right.ad.antipode} Let $\bH=(\ul{\bH},\ol{\bH},\tau)$
be a weak braided Hopf monad on $\A$. Then the functor $K_\omega:\A \to \Aa$
admits a right adjoint $\ol{K}:\Aa \to \A$.
\end{proposition}
\begin{proof} Suppose that $\bH$ has an antipode $S$. By Proposition \ref{pair}, we have to show that
for any $(a,h,\theta) \in \Aa$, the pair (\ref{pair.d}) has an equaliser. We claim that the pair is
cosplit by the composite $d:H(a) \xr{S_a} H(a) \xr{h} a $. Indeed,
  in the proof of \cite[Proposition 3.5 (ii)]{AVR} it is shown that
$$H(h) \cdot \delta_a \cdot e_a \cdot d \cdot \theta =\theta \cdot d \cdot \theta.$$
 It remains to prove that $$d \cdot H(h) \cdot \delta_a \cdot e_a=1.$$
For this, consider the diagram
$$ \xymatrix  @R=.4in @C=.5in{
a \ar[rrdd]_-{e_a} \ar[r]^-{e_a} &  H(a) \ar[rdd]^(.40){\ol{\xi}_a} \ar[r]^-{\delta_a} &
HH(a) \ar[d]^{S_{H(a)}}\ar[r]^-{H(h)} & H(a) \ar[d]^{S_a}\\
& &HH(a) \ar[r]^-{H(h)} \ar[d]^{m_a} & H(a) \ar[d]^{h}\\
&& H(a) \ar[r]_-{h} & a,}
$$ in which
\begin{itemize}
  \item the left triangle commutes by Proposition \ref{AVR} (1);
  \item the right triangle commutes because $S$ is an antipode;
  \item the top square commutes by naturality of $S$;
  \item the bottom square commutes since $(a,h)\in \A_{\ul{\bH}}$.
\end{itemize}
So the whole diagram commutes and since $h \cdot e_a =1$,
the outer paths show  that the desired equality holds.

Thus the composite $ d\cdot \theta: a \xr{\theta} H(a) \xr{S_a} H(a) \xr{h} a$ is an idempotent
and if $a \xr{q_a} \ol{a} \xr{\iota_a} a$ is a splitting of this idempotent, then
the diagram
\begin{equation*}\label{split.eq}
\xymatrix{\ol{a} \ar[r]^{\iota_a} & a \ar@/^2pc/@{->}[rrr]^{\theta}
\ar[r]_{e_a} & H(a) \ar[r]_{\delta_a} & HH(a) \ar[r]_{H(h)}& H(a)}
\end{equation*}
is a (split) equaliser in $\A$. Thus, $\ol{K}$ exists and for any
$(a,h,\theta) \in \Aa$, $\ol{K}(a,h,\theta)=\ol{a}$.
\end{proof}

Dual to Proposition \ref{right.ad.antipode}, we observe:

\begin{proposition}\label{right.ad.antipode1} Let $\bH=(\ul{\bH},\ol{\bH},\tau)$
be a weak braided Hopf
monad on $\A$. Then the functor $K_\omega:\A \to \Aa$
admits a left adjoint $\ul{K}:\Aa \to \A$ that takes $(a,h,\theta) \in \Aa$ to
the object $\ol{a}$ which splits the idempotent $a \xr{\theta} H(a) \xr{S_a} H(a) \xr{h} a$.
Moreover, the diagram
$$
\xymatrix{H(a) \ar@/^2pc/@{->}[rrr]^-{h}
\ar[r]_-{H(\theta)} & HH(a) \ar[r]_-{m_a} & H(a) \ar[r]_-{\ve_a}& a \ar[r]_-{q_a}&\ol{a}}
$$
is a (split) coequaliser.
\end{proposition}

\begin{proposition}\label{split} Let $\bH=(\ul{\bH},\ol{\bH},\tau)$
be a weak braided Hopf monad on a Cauchy complete category $\A$.
Then, for any $(a,h,\theta)\in \Aa$, the diagram
$$\xymatrix{a \ar[r]^{\theta\quad} \ar@{=}[d]& H(a)\ar[d]^{H(\beta_a)}\\
a & H(a) \ar[l]^{h\quad}\,,}
$$
with the idempotent $\beta_a:a \xr{\theta}H(a) \xr{S_a}H(a) \xr{h}a$ (see  proof of Proposition
\ref{right.ad.antipode}) is commutative.
\end{proposition}
\begin{proof} We compute
\begin{align*}
  h \cdot H(\beta_a) \cdot \theta & =h \cdot H(h) \cdot H(S_a) \cdot H(\theta)  \cdot \theta\\
 &= h \cdot m_a \cdot H(S_a) \cdot \delta  \cdot \theta\\
 & =h \cdot \ve_{H(a)} \cdot \omega_a \cdot e_{H(a)} \cdot \theta\\
 & =\ve_a \cdot H(h) \cdot \omega_a \cdot H(\theta) \cdot e_a\\
 & =\ve_a \cdot \theta \cdot h  \cdot e_a=1.
\end{align*}
The second and sixth equations hold since $(a,h)\in \A_{\ul{\bH}}$
and $(a, \theta)\in \A^{\ol{\bH}}$, the third one holds by the definition of an antipode,
the forth one
by naturality of $e$ and $\ve$, and the fifth one by the fact that $(a,h,\theta)\in \Aa$.
\end{proof}

\begin{proposition}\label{counit.1} Let $\bH=(\ul{\bH},\ol{\bH},\tau)$ be a
weak braided Hopf monad on a Cauchy complete category $\A$. Then
\begin{zlist}
\item $K_\omega :\A \to \Aa$ admits both a left and a right adjoint
$\ol{K}, \ul{K}:\Aa \to \A$;
\item the unit $\ul{\eta}: 1 \to K\ul{K}$ of
$\ul{K} \dashv K$ is a split monomorphism, while the counit $\ol{\ve}: K\ol{K} \to 1$ of
 $K \dashv \ol{K}$ is a split epimorphism.
\end{zlist}
\end{proposition}
\begin{proof} According to Propositions \ref{right.ad.antipode} and \ref{right.ad.antipode1},
$K_\omega :\A \to \Aa$ admits both left and right adjoints $\ul{K}, \ol{K}$
and they can be constructed as follows.
For any object $(a,h,\theta)\in \Aa$,   $\beta_a: a \xr{\theta}H(a) \xr{S_a}H(a) \xr{h}a$
is idempotent (see proof of Proposition \ref{right.ad.antipode}) with a splitting
  $a \xr{q_a}\ol{a}\xr{\iota_a} a$ and
   $$\ul{K}(a,h,\theta)=\ol{K}(a,h,\theta)=\ol{a}.$$
   Moreover,
$$\xymatrix{H(a) \ar@/^1.8pc/@{->}[rrr]^-{h}
\ar[r]_-{H(\theta)} & HH(a) \ar[r]_-{m_a} & H(a) \ar[r]_-{\ve_a}& a \ar[r]_-{q_a}&\ol{a}}
$$ is the defining coequaliser diagram for $\ul{K}(a,h,\theta)$, while
$$\xymatrix{\ol{a} \ar[r]^{\iota_a} & a \ar@/^1.8pc/@{->}[rrr]^{\theta}
\ar[r]_{e_a} & H(a) \ar[r]_{\delta_a} & HH(a) \ar[r]_{H(h)}& H(a)}
$$
is the defining equaliser diagram for $\ol{K}(a,h,\theta)$.

It is easy to verify directly, using (\ref{unit}) and (\ref{counit.0}), that
$$\ul{\eta}\,_{(a,\,h,\,\theta)}=H(q_a) \cdot \theta \,\,\, \text{ and } \,\,\,
\ol{\ve}_{(a,\,h,\,\theta)}=h \cdot H(\iota_a).$$

Now, in view of Proposition \ref{split}, we compute
$$\ol{\ve}_{(a,\,h,\,\theta)}\cdot \ul{\eta}\,_{(a,\,h,\,\theta)}
= h \cdot H(\iota_a) \cdot H(q_a) \cdot \theta=
h \cdot H(\beta_a) \cdot \theta=1.$$
Thus $\ol{\ve}_{((a,\,h),\,\theta)}$ is a split epimorphism, while
$\ul{\eta}\,_{((a,\,h),\,\theta)}$ is a split monomorphism.
\end{proof}
\smallskip

We are now ready to state and prove our main result. It subsumes the original version
of the Fundamental Theorem proved for
weak Hopf algebras over fields in \cite[Theorem 3.9]{BNS} as well as various generalisations,
 for example,
for algebras over commutative rings in \cite[Theorem 5.12]{W-weak},  \cite[Theorem 36.16]{BW-Cor},
for Hopf algebroids in \cite[Theorem 4.14]{Bo-HA},
and for weak braided Hopf algebras on monoidal categories \cite[Proposition 3.6]{AVR}.

\begin{thm} \label{main} {\bf Fundamental Theorem.}
Let $\bH=(\ul{\bH},\ol{\bH},\tau)$ be a weak braided bimonad on a
Cauchy complete category $\A$, and $\bT$ and $\bG$ the monad and comonad induced on
$\A^{\ol{\bH}}$  and on $\A_{\ul{\bH}}$, respectively.
Then the following are equivalent:
  \begin{blist}
  \item  $\bH$ is a weak braided Hopf monad;
  \item  the functor  $K_\omega :\A \to \Aa$ admits both left and right adjoints,
         and the right adjoint is monadic;
  \item  the functor $K_\omega :\A \to \Aa$ admits both left and right adjoints,
         and the left adjoint is comonadic;
  \item  the induced natural transformation
  $\gamma:H \!\otimes_{\ul{\bH}^{\,\ol{\xi}}} \!H \to \ol{G}$ is an isomorphism;
  \item 
  the induced natural transformation
  $\gamma\,':\ol{T} \to H \!\ot^{\ol{\bH}\,_{\ol{\xi}}}\!H$ is an isomorphism.
\end{blist}
\smallskip

Moreover, if the (equivalent) conditions above hold, then there is an equivalence of categories
$\A_{{\ul\bH}^{\,\oxi}} \simeq \Aa$.
\end{thm}
\begin{proof} (a) $\Rightarrow$ (b). If $\bH$ is a weak braided Hopf monad, then,
by Propositions \ref{right.ad.antipode} and \ref{right.ad.antipode1}, the functor
$K_\omega: \A \to \Aa$ admits a
right adjoint $\ol{K}$  and a left adjoint $\ul{K}$. So it remains to prove that $\ol{K}$
is monadic. By Proposition \ref{counit.1},
the counit of the adjunction $K \dashv \ol{K}$ is a split epimorphism.
Moreover, since $\A$ is assumed to be Cauchy complete, so is $\Aa$ by Proposition \ref{idem}.
Applying now the dual of \cite[Proposition 3.16]{Me} gives that $\ol{K}$ is monadic.

(b) $\Rightarrow$ (d).
$\ul{\bH}^{\ol{\xi}}$ is separable Frobenius by Proposition \ref{sep.frobenius2}, and hence, by
Proposition \ref{sep.frobenius}, the change-of-base functor
$(\iota^{\ol{\xi}})_! : \A_{\ul{\bH}^{\ol{\xi}}} \to \A_\bH$ exists.
It then follows from Proposition \ref{left.adjoint} that
the comparison functor
$K_{\ul{\bH}^{\ol{\xi}}}:(\A_{\ul{\bH}})^{\bG} \to \A_{\ul{\bH}^{\ol{\xi}}}$ admits a
left adjoint $L_{\ul{\bH}^{\ol{\xi}}}: \A_{\ul{\bH}^{\ol{\xi}}} \to (\A_{\ul{\bH}})^{\bG} $
 such that $U^\bG L_{\ul{\bH}^{\ol{\xi}}}=(\iota^{\ol{\xi}})_!$.
The situation is illustrated by the   diagram
\begin{equation}\label{the situation}
\xymatrix @C=.5in @R=.4in{
\A \ar@/_1pc/@{->}[rr]_{\phi_{\ul{\bH}^{\ol{\xi}}}}\ar@/_4pc/@{->}[rrrr]_(.6)
{K}\ar[rrddd]_{\phi_{\ul{\bH}}} & &\A_{\ul{\bH}^{\ol{\xi}}} \ar[ddd]_{(\iota^{\ol{\xi}})_!}
\ar@/_1pc/@{->}[rr]_{L_{\ul{\bH}^{\ol{\xi}}}}\ar[ll]_{U_{\ul{\bH}^{\ol{\xi}}}}& &
(\A_{\ul{\bH}})^{\bG} \ar@/_3pc/@{->}
[llll]_{\ol{K}}  \ar[llddd]^{U^{\bG}}\ar[ll]_{K_{\ul{\bH}^{\ol{\xi}}}}\\\\&&&&\\& &
 \A_{\ul{\bH}} & . &}
\end{equation}

Next, from the proof of Proposition \ref{sep.frobenius} we know that the defining coequaliser
of $(\iota^{\ol{\xi}})_!$,
\begin{equation}\label{base.1}
\xymatrix{ 
\phi_{\ul{\bH}} H^{\ol{\xi}} U\!_{\ul{\bH}^{\ol{\xi}}}=
\phi_{\ul{\bH}} U_{\ul{\bH}^{\ol{\xi}}} \phi_{\ul{\bH}^{\ol{\xi}}} U\!_{\ul{\bH}^{\ol{\xi}}}
 \ar@{->}@<0.5ex>[rr]^-{\phi_{\ul{\bH}}\,U\!_{\ul{\bH}^{\ol{\xi}}}
\,\,\ve\!\!_{{\ul{\bH}^{\ol{\xi}}}} }
\ar@{->}@<-0.5ex>[rr]_-{\varrho U\!_{\ul{\bH}^{\ol{\xi}}}}&&
\phi_{\ul{\bH}} U\!_{\ul{\bH}^{\ol{\xi}}}\ar[r]^q & (\iota^{\ol{\xi}})_!\,,}
\end{equation}
 where $\rho$ is the composite $\phi_{\ul{\bH}} H^{\ol{\xi}} \xr{\phi_{\ul{\bH}}\,\iota^{\ol{\xi}}} \phi_{\ul{\bH}} H=\phi_{\ul{\bH}} U_{\ul{\bH}}
 \phi_{\ul{\bH}}\xr{\ve_{\ul{\bH}} \,\phi_{\ul{\bH}}} \phi_{\ul{\bH}}$,
is absolute (i.e., is preserved by any functor).
It then follows from Remark \ref{rightiso.3} (2) that the tensor product
$H \!\otimes_{\ul{\bH}^{\ol{\xi}}} \!H$
is just the functor $U_{\ul{\bH}}(\iota^{\ol{\xi}})_! (\iota^{\ol{\xi}})^* \phi_{\ul{\bH}}$.
Moreover, $l=U_{\ul{\bH}} \, q (\iota^{\ol{\xi}})^* \phi_{\ul{\bH}}$.

Since $U_{\ul{\bH}\,^{\ol{\xi}}}\,(\iota^{\ol{\xi}})^*=U_{\ul{\bH}}$, it follows
from Proposition \ref{can.epi}
and Equation (\ref{composition}) that the diagram
$$
\xymatrix @R=.3in @C=.5in{ \phi_{\ul{\bH}}U_{\ul{\bH}\,^{\ol{\xi}}}\,(\iota^{\ol{\xi}})^*=
\phi_{\ul{\bH}}U_{\ul{\bH}}
\ar[rd]_{t_K}\ar[r]^-{q(\iota^{\ol{\xi}})^*}&
(\iota^{\ol{\xi}})_!(\iota^{\ol{\xi}})^*\ar[d]^{t_{L\!_{\ul{\bH}\,
^{\ol{\xi}}}} }\\&  \widetilde{G}}$$ commutes, and since $U_{\ul{\bH}}\widetilde{G}=G$,
 $G\phi_{\ul{\bH}}=\overline{G}$,
and (\ref{base.1}) is absolute, the diagram
$$
\xymatrix @R=.3in @C=.6in{HH \ar[rd]_{U_{\ul{\bH}}\,t_K \phi_{\ul{\bH}}}\ar[r]^{l} &
 H \!\otimes_{\ul{\bH}^{\ol{\xi}}}
\! H  \ar[d]^{U_{\ul{\bH}}\, t_{L\!_{\ul{\bH}\,^{\ol{\xi}}}}\,\, \phi_{\ul{\bH}}}\\  & \ol{G}}
$$ commutes.
Since $l$ is an epimorphism and   (\ref{fund}) is also commutative, it follows that
$\gamma=U_{\ul{\bH}} \,t_{L\!_{\ul{\bH}\,^{\ol{\xi}}}}\, \phi_{\ul{\bH}}.$

Now, since $\ol{K}$ is assumed to be monadic, the comparison functor
$K_{\ul{\bH}^{\ol{\xi}}} : \Ab \to \A_{\ul{\bH}^{\ol{\xi}}}$ is
an equivalence, and hence its left adjoint
$L_{\ul{\bH}^{\ol{\xi}}} : \A_{\ul{\bH}^{\ol{\xi}}} \to \Ab $ is also an equivalence.
Applying Theorem  \ref{fun.th.} to the right commutative triangle in Diagram (\ref{the situation})
 gives that
$t_{L_{\ul{\bH}^{\ol{\xi}}}}$ is an isomorphism. Quite obviously,
$\gamma=U_{\ul{\bH}} t_{L_{\ul{\bH}^{\ol{\xi}}}} \phi_{\ul{\bH}}$ is then also an isomorphism.

(d) $\Rightarrow$ (a). Suppose that
the natural transformation $\gamma:H \!\otimes_{\ul{\bH}^{\ol{\xi}}} \!H \to \ol{G}$ is an
isomorphism. Then we claim that the composite
$$S: H \xr{He} HH \xr{\ol{p}}\ol{G} \xr{\gamma^{-1}} H\!\otimes_{\ul{\bH}^{\ol{\xi}}}
\!H \xr{\widetilde{q}} H$$
is an antipode for $\bH$. Note first that by Propositions \ref{righ.1}, \ref{morphism}, and
Corollary \ref{corollary.1}, the composite
$$\widetilde{q_1}:HH \xr{\ol{p}}\ol{G} \xr{\gamma^{-1}}H\!\otimes_{\ul{\bH}^{\ol{\xi}}}
\!H \xr{\widetilde{q}}H $$ is a morphism of right $\ul{\bH}$-modules.
To show that
$S*1=\ol{\xi}$, consider the diagram
$$
\xymatrix @C=0.4in @R=0.4in{H \ar[r]^-{\delta} \ar[d]_{He} &
HH \ar@{=}[dr]\ar[r]^-{HeH}&HHH \ar[d]^{Hm} \ar[r]^-{\ol{p}H} &
\ol{G}H \ar[r]^-{\gamma^{-1}H}& (H\!\otimes_{\ul{\bH}^{\ol{\xi}}}H)H \ar[r]^-{\widetilde{q}H}&
HH \ar[d]^m\\
HH \ar@{}[rr]^{(1)} \ar@/_1.8pc/@{->}[rrrr]_{l}& & HH \ar[r]_-{\ol{p}} &
\ol{G} \ar[r]^-{\gamma^{-1}}& H\!\otimes_{\ul{\bH}^{\ol{\xi}}}H \ar[r]_-{\widetilde{q}}& H}
$$
in which
\begin{itemize}
  \item the rectangle commutes since the $\widetilde{q_1}$ is a morphism of
        right $\ul{\bH}$-modules;
  \item the triangle commutes since $e$ is the unit of the monad $\ul\bH$;
  \item part (1) commutes since $\ol{p}\,\cdot \delta =\gamma \cdot l \cdot He $
        by (\ref{fund0}), and therefore
  $\gamma^{-1} \cdot \,\ol{p}\, \cdot \delta = l \cdot He $.
\end{itemize}
So the whole diagram commutes and we have $S*1=\widetilde{q} \cdot l \cdot He .$

Commutativity of the trapezoid in (\ref{qtilde}) shows
  $\widetilde{q} \cdot l=m \cdot \ol{\xi}H$, and since $\ol{\xi}H \cdot He=He \cdot \ol{\xi}$
  by naturality of composition, we get
$$
 S*1 =\widetilde{q} \cdot l \cdot He
  =m \cdot \ol{\xi}H \cdot He
 = m \cdot He \cdot \ol{\xi}
   =  \ol{\xi},
$$
and from (\ref{ent-equal}) we conclude $1*S*1= 1*\oxi = 1$.

In the diagram
$$
\xymatrix @C=0.6in @R=0.4in{H \ar[r]^-{\delta} \ar[dd]_{He} &
HH \ar[d]^{H\xi}\ar[r]^-{HeH}&HHH \ar[d]^{HH\xi}
\ar[r]^-{\widetilde{q_1}H}& HH \ar[d]^{H\xi}\\
\ar@{}[r]|{(1)} & HH \ar@{=}[dr]\ar[r]^-{HeH}&HHH \ar[d]^{Hm} \ar[r]^-{\widetilde{q_1}H}&
HH \ar[d]^m\\
HH \ar[rr]_\kappa & & HH \ar[r]_-{\widetilde{q_1}}& H}
$$
\begin{itemize}
 \item part (1) commutes (see (\ref{c-diagrams}) and (\ref{ent})(ii),
  \item the top squares commute by naturality of composition,
  \item the bottom square commutes since $\widetilde{q_1}$ is a morphism of right $\ul{\bH}$-modules, and
  \item the triangle commutes since $e$ is the unit of the monad $\ul{\bH}$.
\end{itemize}
Now, commutativity of the diagram and  $\ol{p}\cdot \kappa=\ol{p}\cdot \ol{i} \cdot \ol{p}=\ol{p}$
 give
\begin{align*}
 S*\xi & =m \cdot H\xi \cdot \widetilde{q_1}H \cdot HeH \cdot\delta
 =\widetilde{q_1} \cdot \kappa \cdot He \\
 & = \widetilde{q} \cdot \gamma^{-1}\cdot \ol{p}\cdot \kappa \cdot He
 = \widetilde{q} \cdot \gamma^{-1}\cdot \ol{p}\cdot He
 =S.
\end{align*}

Recalling $1*\oxi=1=\xi* 1$ (from (\ref{ent-equal}), (\ref{c-diagrams})) yields
$$1*S*1 = 1*\oxi = 1 = \xi * 1$$
and applying Lemma \ref{convolution} shows
  $1*S=1*S*\xi=  \xi*\xi =\xi$ and eventually
  $$S*1*S=S*\xi=S , \quad  \ol{\xi}*S=S*1*S=S.$$
Thus, $S$ is an antipode and hence $\bH$ is a weak braided Hopf monad, as asserted.
\smallskip

The implications (a)$\Rightarrow$(c)$\Rightarrow$(e)$\Rightarrow$(a)  hold by symmetry and
 the last assertion follows from the proof of the implication (b)$\Rightarrow$(d).
\end{proof}

\begin{remark} \em When a weak braided bimonad
$\bH$ has an antipode $S$, it can be shown that the composite
$$ \ol{G}\xr{\ol{i}} HH \xr{\delta H} HHH \xr{HSH} HHH\xr{Hm} HH\xr{l} H
 \!\otimes_{\ul{\bH}^{\ol{\xi}}} \!H$$
is the two-sided inverse of $\gamma:H \!\otimes_{\ul{\bH}^{\ol{\xi}}} \!H \to \ol{G}$,
 while the composite
$$H\!\ot^{\ol{\bH}\,_{\ol{\xi}}}\!H \xr{\text{can}}HH \xr{H\delta}HHH
\xr{HSH}HHH \xr{mH}HH \xr{\ol{p}\,'}\ol{T}$$
is the two-sided inverse of $\gamma':\ol{T} \to H\!\ot^{\ol{\bH}\,_{\ol{\xi}}}\!H$.
\end{remark}

\begin{theorem} \label{main.1} Let $\bH=(\ul{\bH},\ol{\bH},\tau)$ be a
weak braided bimonad on a
Cauchy complete category $\A$, and $\bT$ and $\bG$ the monad and comonad it induces
on $\A^{\ol{\bH}}$  and  $\A_{\ul{\bH}}$, respectively.
\begin{zlist}
  \item  If $H$ preserves the existing coequalisers, the following are equivalent:
  \begin{blist}
  \item   $\bH$ is a weak braided Hopf monad;
  \item   the functor $K_\omega :\A \to \Aa$ admits a monadic right adjoint;
  \item   
  the induced natural transformation
  $\gamma:H \!\otimes_{\ul{\bH}^{\,\ol{\xi}}} \!H \to \ol{G}$ is an isomorphism.
\end{blist}
  \item  If $H$ preserves the existing equalisers, the following are equivalent:
  \begin{blist}
  \item   $\bH$ is a weak braided Hopf monad;
  \item   the functor $K_\omega :\A \to \Aa$ admits a comonadic left adjoint;
  \item  
  the induced natural transformation
  $\gamma':\ol{T} \to H\!\ot^{\ol{\bH}\,_{\ol{\xi}}}\!H$ is an isomorphism.
\end{blist}
\end{zlist}
\smallskip

Moreover, if the (equivalent) conditions above hold, then there is an equivalence of categories
$\A_{\ul{\bH}^{\ol{\,\xi}}} \simeq \Aa$.
\end{theorem}
\begin{proof} By symmetry, it suffices to prove (1). Then, in the light of Proposition \ref{main},
we need only to show that (b) implies (c). So suppose the functor $K_\omega :\A \to \Aa$ to admit
a monadic right adjoint $\ol{K}$. Then the comparison functor $K_{\ul{\bH}^{\ol{\xi}}} : \Ab \to \A_{\ul{\bH}^{\ol{\xi}}}$ is
an equivalence and hence has a left adjoint left inverse
$L_{\ul{\bH}^{\ol{\xi}}} : \A_{\ul{\bH}^{\ol{\xi}}} \to \Ab $. It
then follows from Proposition  \ref{left.adjoint} that
 $U^\bG L_{\ul{\bH}^{\ol{\xi}}}=(\iota^{\ol{\xi}})_!$.

Since $H$ is assumed to preserve the existing coequalisers, the forgetful functor
$U_\bT : \A_\bT \to \A$
also preserves the existing coequalisers (e.g. \cite[Proposition 4.3.2]{FB}).
Since colimits in functor
categories are calculated componentwise, this implies that the image of the
coequaliser (\ref{base.1})
under the functor
$[\A_{\ul{\bH}^{\,\ol{\xi}}},U_\bH]:[\A_{\ul{\bH}^{\,\ol{\xi}}},\A_\bH]\to
[\A_{\ul{\bH}^{\,\ol{\xi}}},\A] $ is again a coequaliser.
Then
 $\gamma=U_{\ul{\bH}}\, t_{L_{\ul{\bH}^{\ol{\xi}}}} \,\phi_{\ul{\bH}}$
by exactly the same argument used in the proof of the implication (b)$\Rightarrow$(d) of
Proposition \ref{main}.

Now, since $L_{\ul{\bH}^{\ol{\xi}}}$ is an equivalence of categories and
$U^\bG L_{\ul{\bH}^{\ol{\xi}}}=(\iota^{\ol{\xi}})_!$,
it follows by Theorem  \ref{fun.th.} that $t_{L_{\ul{\bH}^{\ol{\xi}}}}$
(and hence also $\gamma$) is an isomorphism.
Thus (b) implies (c), as required.
\end{proof}

\smallskip

{\bf Acknowledgments.}
The first author gratefully acknowledges the support by the
Shota Rustaveli National Science Foundation Grants DI/18/5-113/13 and FR/189/5-113/14.

\end{document}